\documentclass[11pt]{article}
\usepackage[utf8]{inputenc}
\usepackage[T1]{fontenc}
\usepackage{lmodern}
\usepackage{amsmath,amsfonts,amssymb,amsthm}
\usepackage{graphicx}
\usepackage[noend]{algorithmic}
\usepackage{booktabs,array,tabularx}
\usepackage{float}
\usepackage{adjustbox}
\usepackage{microtype}
\usepackage[authoryear,round]{natbib}
\usepackage[a4paper,top=3cm,bottom=3cm,left=3cm,right=3cm]{geometry}
\usepackage[bookmarksnumbered,colorlinks,citecolor=blue,linkcolor=blue,
urlcolor=blue,breaklinks,linktocpage]{hyperref}
\usepackage[all]{hypcap}
\usepackage{enumitem}
\usepackage{placeins}
\usepackage{needspace}
\theoremstyle{plain}
\newtheorem{theorem}{Theorem}[section]

\newtheorem{proposition}[theorem]{Proposition}

\theoremstyle{definition}
\newtheorem{definition}[theorem]{Definition}
\newtheorem{algo}[theorem]{Algorithm}
\theoremstyle{remark}
\newtheorem{remark}[theorem]{Remark}
\newtheorem{example}[theorem]{Example}
\newcommand{\Z}{\mathbb{Z}}
\newcommand{\GL}{\operatorname{GL}}
\newcommand{\AffGL}{\overrightarrow{\GL}}

\newcommand{\pc}{\operatorname{pc}}
\newcommand{\bx}{\begin{bmatrix}}
\newcommand{\ex}{\end{bmatrix}}

\allowdisplaybreaks[1]

\renewcommand{\arraystretch}{1.12}
\begin{document}

\title{The art of counterpoint:\\
A Mazzola-type model of three-voice first-species counterpoint}
\author{
Ren Okumura\thanks{Corresponding author. Email: 25rmd09@ms.dendai.ac.jp.
This work was supported by the Research Institute for Science and
Technology of Tokyo Denki University, Japan (grant number Q26K-01).},\quad Takuro Shibayama\\
Department of Informatics, Tokyo Denki University, Tokyo, Japan
}
\date{29 September 2026}
\maketitle

\begin{abstract}
In this paper, we extend Mazzola’s model of two-voice counterpoint to three-voice first-species counterpoint. The construction combines a fiber product over a shared lower voice with a harmonic mask and a two-stage maximization defining admitted successors. For the Fuxian dichotomy, we compute the successor relation and investigate connections with the Riemann dichotomy and neo-Riemannian transformations. Among pairs of same-mode triads, the model admits the most transporter realizations exactly at the pairs that generate Mazzola's Riemann monoid, but it does not single out the dominant–tonic pair, and it admits only 12 of the 192 parsimonious neo-Riemannian realizations, largely because it excludes transitions that keep a pair of voices stationary.

\end{abstract}

\medskip


\section{Introduction}
Counterpoint and harmony offer complementary perspectives on the organization of musical voices and sonorities. This paper is motivated by the broader question of how their mathematical descriptions can be connected within a common framework. Mathematical approaches to this relationship already include geometrical theories of chords and voice leading, which distinguish between musical representations according to the information they retain about pitch, register, and voice assignment \citep{CallenderQuinnTymoczko2008,TymoczkoGeometry2011}. \citet[Chapters~1 and~3]{TymoczkoTonality2023} develops this line of work into a hierarchical account of tonality, in which voices move along chords and chords move along scales, and examines the relationship between harmony and counterpoint in triadic textures. Neo-Riemannian theory provides a related transformational perspective on harmonic relations, particularly through operations on major and minor triads and their parsimonious voice-leading realizations \citep{Cohn1997,Cohn1998,CransFioreSatyendra2009}. Within this setting, we investigate the relationship between a three-voice extension of Mazzola’s algebraic counterpoint model and transformations of triadic sonorities.

Mazzola’s model of first-species counterpoint defines admitted successors of two-voice intervals through consonance--dissonance dichotomies, affine symmetries, and a maximization principle \citep{Mazzola2017,AgustinAquinoJunodMazzola2015}. Its algebraic formulation supports extensions beyond the twelve-tone Fuxian setting. Existing developments include counterpoint in \(2k\)-tone equal temperament \citep{AgustinAquino2009}, results over arbitrary rings \citep{AriasValeroAgustinAquinoLluisPuebla2021}, variations and evaluations of the successor mechanism \citep{AriasValeroAgustinAquinoLluisPuebla2023}, and a projection-oriented model of second-species counterpoint \citep{AgustinAquinoMazzola2022}. Our starting point is the two-voice admitted-successor mechanism, which we extend by coordinating the successor mechanisms of the different pairs of voices.

This approach differs from encoding contrapuntal or harmonic prescriptions directly as constraints. Constraint-based approaches have an established history in musical harmonization \citep{PachetRoy2001}, and \citet{Lamotte2024} develops a constraint-programming formalization of Fux’s three-voice counterpoint. For a given cantus firmus, \citet{KomosinskiSzachewicz2015} evaluate first-species counterpoints by several criteria and use the dominance relation of multicriteria analysis to determine the nondominated ones, without aggregating the criteria into a single score. Such approaches provide important points of comparison, but our principal question concerns the construction of a successor relation from Mazzola’s symmetries and maximization principles. We ask how these pairwise mechanisms can be combined in a three-voice setting and how the resulting relation interacts with the harmonic structure of simultaneous sonorities.

The harmonic side of this question also has direct algebraic precedents. Noll studies triads through the monoid of affine maps that preserve a triad \citep{Noll2005}. Fiore and Noll investigate the relationship between the triadic monoid and neo-Riemannian group actions, and subsequently develop transformations of triads with specified voicings \citep{FioreNoll2011,FioreNoll2018}. These studies provide a context for distinguishing transformations of pitch-class collections from transformations that retain voice assignments. A further connection appears in Mazzola’s correspondence between the Riemann dichotomy and two-voice consonant intervals \citep[Section~30.2.1]{Mazzola2017}; see also \citet[Section~3.5]{AgustinAquinoJunodMazzola2015}. That correspondence is between sets and does not itself specify which three-voice realizations are admitted by a contrapuntal successor relation. We investigate this additional question within the model constructed here.

Three voices provide a natural setting for this investigation because they allow the three members of a triad to be assigned to separate voices. We represent a sonority by a designated lower-voice pitch class and the two directed interval classes measured from it. The use of a bass reference is conceptually related to figured bass, whose theoretical significance extends beyond its practical role in accompaniment \citep{Christensen2010}. Unlike figured bass, however, our representation uses interval classes rather than traditional figures. It distinguishes the roles of the three voices but does not retain their registral placement. It accommodates both complete triads and sonorities with pitch-class doublings, and the interval class between the upper voices is recovered as the difference of the two bass-referenced interval classes.

We construct the ambient three-voice space as a fiber product of two dual-number spaces over their shared lower-voice coordinate. A harmonic mask then specifies which combinations of intervals are admissible as vertical sonorities. To define transitions, we apply the ordinary two-voice maximization to the active pairwise projections and combine the resulting symmetries through a second, joint maximization. This construction separates the choice of admissible sonorities from the mechanism that selects their successors, without requiring a strong dichotomy on the full three-voice space. After developing the construction over \(\mathbb Z_{2k}\), we specialize it to the Fuxian dichotomy in twelve-tone equal temperament and compute the resulting finite successor relation.

The musical interpretation of such a relation requires a separate assessment. The relationship between Mazzola’s model and traditional contrapuntal prescriptions has been debated, particularly with respect to pitch-class reduction, the treatment of fifths and octaves, and the interpretation of agreement between a model and a reference rule system \citep{TymoczkoCounterpoint2011,AgustinAquinoMazzola2011}. Formalizing the reference rules also involves choices, since contrapuntal treatises differ in their prescriptions and in the extent to which those prescriptions are explicit \citep{SamoryEtAl2014}. These considerations motivate an evaluation that distinguishes the algebraic construction from the musical classification against which it is compared.

For this purpose, we construct a descriptive classification of local three-voice first-species progressions, informed by Fux and Jeppesen \citep{Fux1965,Jeppesen1992} and extending the two-voice strict-style analysis of \citet{AriasValeroAgustinAquinoLluisPuebla2023}. We compare the model with this classification at both the integer-pitch and reduced pitch-class levels. We also compare the full construction with a variant that omits joint maximization and examine the sensitivity of the evaluation to alternative rule choices. The assessment concerns agreement with the specified local classification; it does not constitute a complete account of first-species composition or an evaluation of musical quality. Comparing the integer-pitch and pitch-class levels also shows which limitations result from the loss of registral information.

Finally, we examine the relation between the three-voice construction and triadic transformations. We classify the ordered pairs of distinct major or minor triads whose affine transporters generate the Riemann monoid. We then realize Mazzola’s correspondence through an embedding into the three-voice state space followed by the lower--middle projection. This allows us to investigate which voice-assigned transporter realizations are admitted by the model. We separately examine parsimonious realizations of the neo-Riemannian transformations \(\mathsf P\), \(\mathsf L\), and \(\mathsf R\). The results exhibit structural connections between the two frameworks while also identifying restrictions inherited from the two-voice successor mechanism. In particular, the exclusion of transitions that leave an active pair of voices stationary constrains the admission of realizations that preserve common tones.

Section~\ref{sec:construction} defines the three-voice construction, and Section~\ref{sec:fuxian} specializes it to the Fuxian dichotomy. Section~\ref{sec:strict3} develops the strict-style reference classification, and Section~\ref{subsec:evaluation} evaluates the successor relation against it. Section~\ref{sec:noll} investigates the connections with the Riemann dichotomy, affine transporters, and neo-Riemannian transformations.

\section{A three-voice extension}\label{sec:construction}
\subsection{Pitch-class representation and harmonic masks}

Let \(\Z_{2k}\) be the ring of integers modulo \(2k\), which we use to
represent pitch classes in \(2k\)-tone equal temperament. Intervals are
represented by directed differences in this ring, so all interval
arithmetic is carried out modulo \(2k\). A marked interval dichotomy
is a partition
\[
(X/Y),\qquad Y=\Z_{2k}\setminus X,\qquad |X|=|Y|=k,
\]
in which \(X\) is designated as the consonance set and \(Y\) as the
dissonance set. The affine group
\[
\AffGL(\Z_{2k})
=\{T_u\cdot v:u\in\Z_{2k},\ v\in\Z_{2k}^{\times}\},
\qquad (T_u\cdot v)(z)=u+vz,
\]
acts on interval dichotomies. A dichotomy is \emph{autocomplementary}
if some affine symmetry maps \(X\) onto \(Y\), and \emph{rigid} if the
only affine symmetry that leaves \(X\) invariant is the identity. It is
\emph{strong} if it is both autocomplementary and rigid. Equivalently,
\((X/Y)\) is strong if and only if there is a unique affine symmetry
\(p=T_u\cdot v\) such that \(p(X)=Y\); this symmetry is called the
\emph{polarity} of the dichotomy
\citep{AgustinAquino2009,AriasValeroAgustinAquinoLluisPuebla2021}.

We represent three-voice sonorities as elements of the square-zero
extension
\[
\Z_{2k}[\varepsilon_1,\varepsilon_2]
=\{a+\varepsilon_1. b+\varepsilon_2. c:a,b,c\in\Z_{2k}\},
\qquad
\varepsilon_i\varepsilon_j=0
\quad(i,j\in\{1,2\}),
\]
and write
\[
\begin{bmatrix}c\\b\\a\end{bmatrix}
:=a+\varepsilon_1. b+\varepsilon_2. c.
\]

The lower, middle, and upper voices have pitch classes \(a\), \(a+b\),
and \(a+c\), respectively. The directed intervals from the lower to the
middle voice, from the lower to the upper voice, and from the middle to
the upper voice are \(b\), \(c\), and \(c-b\), respectively. The voice
names indicate the intended roles of the voices; the representation
does not encode their registral order. An interval is called
\emph{bass-referenced} if it is measured from the designated lower
voice (the bass), which need not carry the root of the chord.

\begin{definition}
A \emph{three-voice harmonic mask} associated with \(X\) is a subset
\(H\subseteq X\times X\). We write its elements \((b,c)\) as columns
\[
\begin{bmatrix}c\\b\end{bmatrix},\qquad b,c\in X.
\]
The corresponding set of admissible sonorities, or \emph{states}, is
\[
X_H=\left\{\begin{bmatrix}c\\b\\a\end{bmatrix}:
a\in\Z_{2k},\
\begin{bmatrix}c\\b\end{bmatrix}\in H\right\}.
\]
\end{definition}

For an admissible sonority in \(X_H\), both bass-referenced intervals
belong to \(X\), but the middle--upper interval need not belong to \(X\).

\subsection{Two-voice preliminaries}
Let \((X/Y)\) be a strong dichotomy with polarity \(p=T_u\cdot v\).
Write \(\Z_{2k}[\varepsilon]\) for the ring of dual numbers over
\(\Z_{2k}\), where \(\varepsilon^2=0\). The induced counterpoint
dichotomy is
\[
X[\varepsilon]=\Z_{2k}+\varepsilon. X,\qquad
Y[\varepsilon]=\Z_{2k}+\varepsilon. Y.
\]
The affine symmetries considered here are the maps
\(z\mapsto t+sz\) with \(t\in\Z_{2k}[\varepsilon]\) and
\(s\in\Z_{2k}[\varepsilon]^\times\); they form the group
\(\AffGL(\Z_{2k}[\varepsilon])\).

For a source \(\zeta=x+\varepsilon. r\in X[\varepsilon]\), the induced
polarity associated with the fiber \(x+\varepsilon. \Z_{2k}\) is
\[
p_x[\varepsilon]=T_{x(1-v)+\varepsilon. u}\cdot v.
\]
Let \(\Omega_X(\zeta)\) be the set of affine symmetries
\(g\in\AffGL(\Z_{2k}[\varepsilon])\) that satisfy the two contrapuntal
conditions preceding the maximization step:
\[
\zeta\in gY[\varepsilon],
\qquad
p_x[\varepsilon]\bigl(gX[\varepsilon]\bigr)=gY[\varepsilon].
\]
Here \(x\) is the pitch class of the \emph{cantus firmus}, the voice
from which the interval \(r\) is measured.
When \(\Omega_X(\zeta)\ne\varnothing\), define
\[
m_X(\zeta)=
\max_{g\in\Omega_X(\zeta)}
|X[\varepsilon]\cap gX[\varepsilon]|.
\]
The maximizing symmetries, also called \emph{contrapuntal symmetries},
form the set
\[
\Sigma_X(\zeta)=
\left\{g\in\Omega_X(\zeta):
|X[\varepsilon]\cap gX[\varepsilon]|=m_X(\zeta)\right\}.
\]
The ordinary two-voice admitted-successor set is
\[
S_X(\zeta)=
\bigcup_{g\in\Sigma_X(\zeta)}
\bigl(X[\varepsilon]\cap gX[\varepsilon]\bigr).
\]
Thus a target must be consonant and must belong to the image of
\(X[\varepsilon]\) under at least one maximizing symmetry.
The restriction to maximizing symmetries is the maximization step of
Hichert's algorithm
\citep{AgustinAquinoJunodMazzola2015,AriasValeroAgustinAquinoLluisPuebla2021}.

\subsection{Three-voice successors}

The pairwise projections
\(\pi_{\mathrm{LM}},\pi_{\mathrm{LU}},\pi_{\mathrm{MU}}\colon
\Z_{2k}[\varepsilon_1,\varepsilon_2]\to\Z_{2k}[\varepsilon]\)
are defined by
\begin{align*}
\pi_{\mathrm{LM}}\begin{bmatrix}c\\b\\a\end{bmatrix}&=a+\varepsilon. b,\\
\pi_{\mathrm{LU}}\begin{bmatrix}c\\b\\a\end{bmatrix}&=a+\varepsilon. c,\\
\pi_{\mathrm{MU}}\begin{bmatrix}c\\b\\a\end{bmatrix}&=(a+b)+\varepsilon. (c-b).
\end{align*}

For \(\xi=\begin{bmatrix}c\\b\\a\end{bmatrix}\in X_H\), define the set of active pairs by
\[
A(\xi)=
\{\alpha\in\{\mathrm{LM},\mathrm{LU},\mathrm{MU}\}:\pi_\alpha(\xi)\in X[\varepsilon]\}.
\]
The pairs \(\mathrm{LM}\) and \(\mathrm{LU}\) are always active, and \(\mathrm{MU}\) is active
if and only if \(c-b\in X\).

Throughout the construction, we assume that
\(\Omega_X(\pi_\alpha(\xi))\ne\varnothing\) for every
\(\xi\in X_H\) and every \(\alpha\in A(\xi)\).
Set
\[
\Sigma_\alpha(\xi)=\Sigma_X(\pi_\alpha(\xi)),
\qquad
\Sigma(\xi)=\prod_{\alpha\in A(\xi)}\Sigma_\alpha(\xi).
\]
For a tuple
\(\mathbf g=(g_\alpha)_{\alpha\in A(\xi)}\in\Sigma(\xi)\),
define the set of compatible three-voice targets by
\begin{equation}\label{eq:compatible-targets}
C_H(\xi;\mathbf g)
=X_H\cap\bigcap_{\alpha\in A(\xi)}
\pi_\alpha^{-1}\bigl(g_\alpha X[\varepsilon]\bigr),
\end{equation}
and let
\[
\mu_H(\xi;\mathbf g)=|C_H(\xi;\mathbf g)|
\]
denote its cardinality.
Define the three-voice maximum and the set of maximizing tuples by
\[
M_H(\xi)=\max_{\mathbf g\in\Sigma(\xi)}\mu_H(\xi;\mathbf g),
\]
\[
\Gamma_H(\xi)=
\{\mathbf g\in\Sigma(\xi):\mu_H(\xi;\mathbf g)=M_H(\xi)\}.
\]
The resulting three-voice admitted-successor set is
\begin{equation}\label{eq:successors}
S_H(\xi)=\bigcup_{\mathbf g\in\Gamma_H(\xi)}
C_H(\xi;\mathbf g).
\end{equation}
A transition \(\xi\to\eta\) is called \emph{admitted} if
\(\eta\in S_H(\xi)\).

The construction thus involves two maximization stages: first over
the symmetries of each active pair separately, and then over tuples of
these maximizing symmetries. It is \emph{fibered} in the sense that
the two bass-referenced projections share the lower-voice coordinate:
via \(\pi_{\mathrm{LM}}\) and \(\pi_{\mathrm{LU}}\), the ring
\(\Z_{2k}[\varepsilon_1,\varepsilon_2]\) is the fiber product of two
copies of \(\Z_{2k}[\varepsilon]\) over \(\Z_{2k}\). The maximization
does not range over affine symmetries of the full three-voice
extension.

Every admitted three-voice transition projects to an ordinary
two-voice admitted transition under both bass-referenced projections:
\begin{equation}\label{eq:projection}
\eta\in S_H(\xi)
\ \Longrightarrow\
\pi_\alpha(\eta)\in S_X(\pi_\alpha(\xi))
\qquad(\alpha=\mathrm{LM},\mathrm{LU}).
\end{equation}
Indeed, \(\pi_\alpha(\eta)\in g_\alpha X[\varepsilon]\) by
\eqref{eq:compatible-targets}, and membership in \(X_H\) guarantees
that \(\pi_\alpha(\eta)\in X[\varepsilon]\) for these two projections.
When \(\mathrm{MU}\) is active, its maximizing symmetries are computed from its
own source interval \(c-b\).
Equation~\eqref{eq:compatible-targets} requires
\(\pi_{\mathrm{MU}}(\eta)\in g_{\mathrm{MU}}X[\varepsilon]\) but does not separately
require \(\pi_{\mathrm{MU}}(\eta)\in X[\varepsilon]\).
The middle--upper target interval may therefore lie outside \(X\),
subject to the mask and the remaining successor conditions.
Since \(\pi_\alpha(\xi)\in g_\alpha Y[\varepsilon]\) for every
\(\mathbf g\in\Sigma(\xi)\) and \(\alpha\in A(\xi)\), no compatible
target keeps an active pair of voices stationary:
\begin{equation}\label{eq:stationary}
\eta\in C_H(\xi;\mathbf g),\ \alpha\in A(\xi)
\ \Longrightarrow\
\pi_\alpha(\eta)\ne\pi_\alpha(\xi).
\end{equation}

The nonemptiness assumption and the finiteness of the affine group
ensure that \(\Sigma(\xi)\) is finite and nonempty, so both maxima are
attained and \(S_H(\xi)\) is well defined. Thus
\[
E_H=\{(\xi,\eta)\in X_H\times X_H:\eta\in S_H(\xi)\}
\]
defines a finite directed graph \((X_H,E_H)\). A vertex need not have
an outgoing edge, however: \(M_H(\xi)=0\), or equivalently
\(S_H(\xi)=\varnothing\), is possible.

\noindent\begin{minipage}{\linewidth}
\begin{algo}[Three-voice successors]\label{alg:three-voice}
Under the nonemptiness assumption above, the following algorithm
computes the admitted-successor set defined in~\eqref{eq:successors}.

\begin{algorithmic}[1]
\REQUIRE A strong dichotomy \((X/Y)\) of \(\Z_{2k}\);
a harmonic mask \(H\subseteq X\times X\); a source \(\xi\in X_H\).
\ENSURE The admitted-successor set \(S_H(\xi)\).
\STATE \(A(\xi)\gets
\{\alpha\in\{\mathrm{LM},\mathrm{LU},\mathrm{MU}\}:\pi_\alpha(\xi)\in X[\varepsilon]\}\).
\FORALL{\(\alpha\in A(\xi)\)}
\STATE \(\Sigma_\alpha(\xi)\gets\Sigma_X(\pi_\alpha(\xi))\).
\ENDFOR
\STATE \(M\gets-\infty,\qquad\Gamma\gets\varnothing\).
\FORALL{\(\mathbf g\in\prod_{\alpha\in A(\xi)}\Sigma_\alpha(\xi)\)}
\STATE \(C_H(\xi;\mathbf g)\gets
\{\eta\in X_H:\pi_\alpha(\eta)\in g_\alpha X[\varepsilon]
\text{ for all }\alpha\in A(\xi)\}\).
\STATE \(\mu\gets|C_H(\xi;\mathbf g)|\).
\IF{\(\mu>M\)}
\STATE \(M\gets \mu,\qquad\Gamma\gets\{\mathbf g\}\).
\ELSIF{\(\mu=M\)}
\STATE \(\Gamma\gets\Gamma\cup\{\mathbf g\}\).
\ENDIF
\ENDFOR
\RETURN \(\bigcup_{\mathbf g\in\Gamma}C_H(\xi;\mathbf g)\).
\end{algorithmic}
\end{algo}
\end{minipage}

\section{Specialization to the Fuxian dichotomy}\label{sec:fuxian}

We now specialize the construction to \(\Z_{12}\).
The Fuxian interval dichotomy is
\[
(K/D)=\bigl(\{0,3,4,7,8,9\}/\{1,2,5,6,10,11\}\bigr).
\]
It is strong \citep{AriasValeroAgustinAquinoLluisPuebla2021}, with polarity
\[
p=T_2\cdot5,\qquad p(x)=2+5x.
\]
The elements \(0,3,4,7,8,9\) of \(K\) represent the unison, the minor
third, the major third, the perfect fifth, the minor sixth, and the
major sixth, respectively, up to octave equivalence.

\begin{definition}[A harmonic mask for the Fuxian dichotomy]\label{def:mask}
Define
\[
H_{\mathrm{Fux}}=
\left\{\begin{bmatrix}c\\b\end{bmatrix}\in K\times K:
\begin{array}{l}
c-b\notin\{\pm1,\pm2\},\\
\{b,c\}\ne\{4,8\}
\end{array}
\right\}.
\]
\end{definition}

The first condition excludes the directed pitch-class intervals
\(1\), \(2\), \(10\), and \(11\) between the two upper voices, that is,
minor and major seconds and sevenths up to octave equivalence.
Among the remaining configurations, the second condition excludes
exactly \(\begin{bmatrix}8\\4\end{bmatrix}\) and \(\begin{bmatrix}4\\8\end{bmatrix}\), the two ordered configurations of
the augmented triad \(\{0,4,8\}\) above a bass of pitch class~0.
This mask is a modeling choice; it does not encode the full set of
Fuxian rules.

The mask has 26 elements. Ten of them yield three distinct pitch
classes:
\[
\{b,c\}\in
\bigl\{\{3,7\},\{4,7\},\{3,8\},\{4,9\},\{3,9\}\bigr\},
\]
with two orderings of each pair. The remaining configurations involve
a pitch-class doubling and form the set
\[
\left\{\begin{bmatrix}t\\0\end{bmatrix},\begin{bmatrix}0\\t\end{bmatrix},\begin{bmatrix}t\\t\end{bmatrix}:t\in K\right\};
\]
since the three forms coincide at \(t=0\), this set has
\(1+3\cdot5=16\) elements.
Unison and octave doublings are indistinguishable at the pitch-class
level. Specializing \(X_H\) to this mask gives the state set
\(K_{H_{\mathrm{Fux}}}\), which has \(12\cdot26=312\) elements.

The pairs \(\{3,7\}\) and \(\{4,7\}\) represent minor and major triads
in root position, and the pairs \(\{3,8\}\), \(\{4,9\}\), and
\(\{3,9\}\) represent major, minor, and diminished triads in first
inversion, respectively.
For example, \(\begin{bmatrix}8\\3\end{bmatrix}\in H_{\mathrm{Fux}}\) represents a major triad
in first inversion above a bass of pitch class~0, with a perfect
fourth between the upper voices. A perfect fourth above the bass is
excluded because \(5\notin K\).

Table~\ref{tab:fuxian-forbidden} summarizes the forbidden successors
in the ordinary two-voice model
\citep{AgustinAquinoJunodMazzola2015}; see also
\citet[Section~8.4]{AriasValeroAgustinAquinoLluisPuebla2021}.
Normalize the lower-voice pitch class of the source to 0, so that the
source is \(\varepsilon. r\) with interval \(r\in K\), and let
\(j\in\Z_{12}\) be the lower-voice pitch class of the target. The table
gives the sets \(F_{r,j}\subseteq K\) of forbidden target intervals;
the admitted targets with lower-voice pitch class \(j\) are therefore
\[
\{j+\varepsilon. \ell:\ell\in K\setminus F_{r,j}\}.
\]
For a source with an arbitrary lower-voice pitch class, \(j\) denotes
the displacement of the lower voice, that is, the target lower-voice
pitch class minus the source lower-voice pitch class, modulo 12.

\begin{table}[H]
\centering
\begin{tabular}{c|l}
\toprule
\(r\)&\(F_{r,j}\)\\
\midrule
0&\(\{0\}\) if \(j\equiv0\pmod6\); \(\varnothing\) otherwise\\
3&\(\{3,9\}\) if \(j\equiv0\pmod3\); \(\varnothing\) otherwise\\
4&\(\{0,4,8\}\) if \(j\equiv0\pmod2\); \(\varnothing\) otherwise\\
7&\(\{7\}\) for every \(j\)\\
8&\(\{8\}\) if \(j\equiv0\pmod6\); \(\varnothing\) otherwise\\
9&\(\{3,9\}\) if \(j\equiv0\pmod3\); \(\varnothing\) otherwise\\
\bottomrule
\end{tabular}
\caption{Forbidden target intervals in the ordinary two-voice
Fuxian model.}\label{tab:fuxian-forbidden}
\end{table}

\begin{proposition}\label{prop:fifths}
Let \(\xi\in K_{H_{\mathrm{Fux}}}\), \(\alpha\in A(\xi)\),
\(\mathbf g\in\Sigma(\xi)\), and
\(\eta\in C_{H_{\mathrm{Fux}}}(\xi;\mathbf g)\). Write
\(\pi_\alpha(\xi)=x+\varepsilon. r\) and
\(\pi_\alpha(\eta)=x'+\varepsilon. r'\). If \(r=7\), then \(r'\ne7\).
In particular, this holds for every \(\eta\in S_{H_{\mathrm{Fux}}}(\xi)\).
\end{proposition}

\begin{proof}
Suppose \(r'=7\). Since \(7\in K\), we have
\(\pi_\alpha(\eta)\in K[\varepsilon]\), and
\(\pi_\alpha(\eta)\in g_\alpha K[\varepsilon]\)
by~\eqref{eq:compatible-targets}. As
\(g_\alpha\in\Sigma_K(\pi_\alpha(\xi))\), it follows that
\(\pi_\alpha(\eta)\in S_K(\pi_\alpha(\xi))\), which contradicts the
row \(r=7\) of Table~\ref{tab:fuxian-forbidden}.
\end{proof}

The argument also covers the pair \(\mathrm{MU}\), because the target interval
7 belongs to \(K\), even though other admitted target intervals of
this pair may lie outside \(K\).
The exclusion is a condition on directed pitch-class intervals; it
does not distinguish a repeated fifth from parallel fifths or from
consecutive fifths by contrary motion (for example, a fifth followed
by a twelfth).

Appendix~\ref{app:successor-summary} labels the 26 configurations of
\(H_{\mathrm{Fux}}\) as \(h_1,\ldots,h_{26}\) (Table~\ref{tab:labels}) and gives
successor counts and a representative example; the supplementary
material lists the complete forbidden-target tables.

\begin{example}
The following sequence of twelve states was constructed by the first
author; the lower voice is adapted from the subject of J.~S.~Bach's
\emph{The Art of Fugue}, BWV~1080
\citep{BachDigitalSourceBWV1080SecondPrint}.
\begin{align*}
&\xi_1=\begin{bmatrix}0\\3\\2\end{bmatrix},\quad
\xi_2=\begin{bmatrix}3\\8\\9\end{bmatrix},\quad
\xi_3=\begin{bmatrix}7\\4\\5\end{bmatrix},\quad
\xi_4=\begin{bmatrix}3\\7\\2\end{bmatrix},\\
&\xi_5=\begin{bmatrix}3\\8\\1\end{bmatrix},\quad
\xi_6=\begin{bmatrix}3\\7\\2\end{bmatrix},\quad
\xi_7=\begin{bmatrix}3\\8\\4\end{bmatrix},\quad
\xi_8=\begin{bmatrix}4\\7\\5\end{bmatrix},\\
&\xi_9=\begin{bmatrix}7\\3\\7\end{bmatrix},\quad
\xi_{10}=\begin{bmatrix}4\\7\\5\end{bmatrix},\quad
\xi_{11}=\begin{bmatrix}3\\8\\4\end{bmatrix},\quad
\xi_{12}=\begin{bmatrix}0\\3\\2\end{bmatrix}.
\end{align*}
All 11 transitions \(\xi_i\to\xi_{i+1}\) are admitted.
Figure~\ref{fig:three-voice-example} shows a realization in which the
octave placement of each voice was chosen by the first author.
\begin{figure}[H]
\centering
\includegraphics[width=\columnwidth]{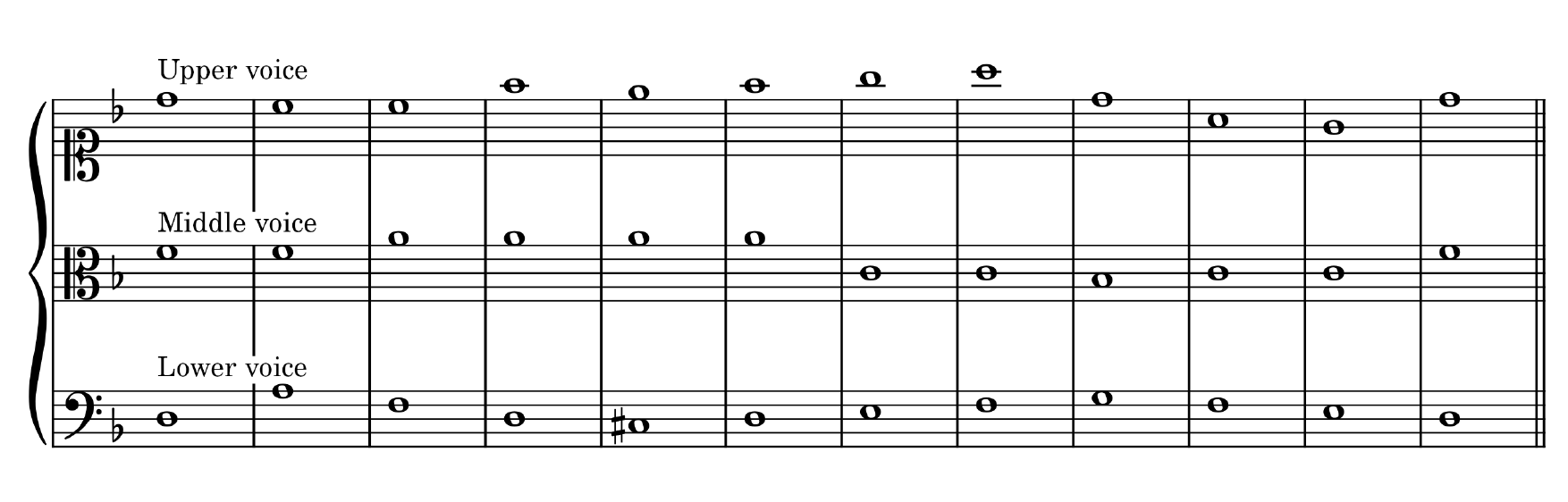}
\caption{A three-voice realization of the sequence
\(\xi_1,\ldots,\xi_{12}\).}
\label{fig:three-voice-example}
\end{figure}
\end{example}

\section{The three-voice strict style}\label{sec:strict3}

We extend the integer-pitch representation and local classification of the
two-voice \emph{strict style} of
\citet{AriasValeroAgustinAquinoLluisPuebla2023} to three-voice first-species
counterpoint, using \citet[pp.~71--85]{Fux1965},
\citet[pp.~175--177]{Jeppesen1992}, and Jeppesen's summary
\citep[pp.~286--289]{Jeppesen1992}. This \emph{descriptive} classification
of adjacent simultaneities provides a reference for assessing an algebraic
successor relation.

\par\medskip\noindent\textbf{Data.}\quad
Pitches and intervals are integers in semitones, without enharmonic
distinctions. Let $\pc\colon\Z\to\Z_{12}$ be the natural projection, which assigns to
each integer pitch or interval its pitch class, and let
$\mathrm{Dia}=\{0,2,4,5,7,9,11\}$ be the diatonic pitch-class set, with
pitch extension $\pc^{-1}(\mathrm{Dia})$. Recall $K=\{0,3,4,7,8,9\}$ and write
\[
\hat K=\pc^{-1}(K),\qquad
\hat K_{\mathrm{perf}}=\pc^{-1}(\{0,7\}),\qquad
\hat K_{\mathrm{imp}}=\pc^{-1}(\{3,4,8,9\})
\]
for the consonant, perfect, and imperfect intervals above the lowest voice.

The lower, middle, and upper voices $\mathrm L$, $\mathrm M$, and $\mathrm U$, with pitches $l$,
$m$, and $u$, define a \emph{three-voice contrapuntal interval}
\[
\xi=l+b\,x_1+c\,x_2,\qquad b=m-l,\quad c=u-l.
\]
This linear polynomial has reduction
\[
\rho(\xi)=\pc(l)+\varepsilon_1\pc(b)+\varepsilon_2\pc(c)
=\begin{bmatrix}\pc(c)\\\pc(b)\\\pc(l)\end{bmatrix}
\]
in the square-zero extension $\Z_{12}[\varepsilon_1,\varepsilon_2]$, where
$\varepsilon_i\varepsilon_j=0$ for $i,j\in\{1,2\}$. Its three
\emph{pair intervals} are
\[
k_{\mathrm{LM}}(\xi)=b,\qquad k_{\mathrm{LU}}(\xi)=c,\qquad k_{\mathrm{MU}}(\xi)=c-b.
\]

A \emph{progression} is a pair $(\xi,\xi')$, with
$\xi'=l'+b'x_1+c'x_2$ and voice \emph{motions}
\[
\Delta_{\mathrm L}=l'-l,\qquad
\Delta_{\mathrm M}=l'+b'-l-b,\qquad
\Delta_{\mathrm U}=l'+c'-l-c.
\]
For $VW\in\{\mathrm{LM},\mathrm{LU},\mathrm{MU}\}$, with $V$ the lower voice, write
$k=k_{VW}(\xi)$ and $k'=k_{VW}(\xi')$. Motion is \emph{similar} if
$\Delta_V\Delta_W>0$ and \emph{contrary} if $\Delta_V\Delta_W<0$.
A \emph{step} has $0<|\Delta|\le2$, a \emph{skip} has $|\Delta|>2$,
and $\Delta=0$ is a repetition.

A composition is a finite sequence $\xi_1,\ldots,\xi_n$, with one note of
equal duration in each voice at each position. We classify its adjacent pairs.

\par\medskip\noindent\textbf{Construction principles.}\quad
Jeppesen introduces three-part writing through exceptions to the two-part
rules \citep[p.~175]{Jeppesen1992}. Fux asks that rules hold between every
pair where possible, allowing departures in the upper voices for a serious
reason \citep[p.~76]{Fux1965}. We extend pairwise conditions to all three
pairs, distinguishing the bass and outer voices where required. Historical
prescriptions motivate, but do not uniquely determine, the classification
\citep[cf.][]{SamoryEtAl2014}.

Registral order, distance bounds, and a common scale delimit the enumeration domain.
Conditions depending on mode, compositional position, or longer passages
are recorded separately in Section~\ref{subsec:prorul3}; a locally good
progression therefore need not be acceptable in every context.
Remark~\ref{rem:variants3} compares the conservative conditions on hidden
perfect consonances between the outer voices with a relaxed version.

\subsection{Preliminary rules}\label{subsec:prerul3}

\begin{enumerate}
\item \emph{Consonance against the lowest voice.} Require $b,c\in\hat K$
\citep[p.~71]{Fux1965}, reckoning consonance from the bass
\citep[p.~112 and n.~4]{Fux1965}; \citep[p.~175]{Jeppesen1992}.
This excludes six-four chords; a diminished triad is admitted only in first
inversion, as a chord of the sixth \citep[p.~176]{Jeppesen1992}.

\item \emph{Intervals between the upper voices.} Perfect and augmented
fourths and diminished fifths are permitted
\citep[p.~175]{Jeppesen1992}; \citep[pp.~112, 131]{Fux1965}, but seconds
and sevenths are excluded:
\[
\pc(c-b)\notin\{1,2,10,11\}.
\]
We exclude Jeppesen's rare first-inversion augmented triad
\citep[p.~176]{Jeppesen1992} by imposing
\[
\{\pc(b),\pc(c)\}\ne\{4,8\}.
\]
This condition is implied by the diatonic condition below; we state it to
fix the harmonic domain explicitly.

\item \emph{Registral order.} $0\le b\le c$, equivalently $l\le m\le u$,
excludes voice crossing and keeps $\mathrm L$ the lowest sounding voice. This is a domain
restriction: Fux permits crossing \citep[p.~36]{Fux1965}, and Jeppesen
recommends it \citep[p.~113, n.~3]{Jeppesen1992}.

\item \emph{Distances.} The enumeration bounds $c\le28$ and $c-b\le16$
limit the outer voices to a major seventeenth and the upper voices to a
major tenth. Jeppesen prefers a distance of at most a tenth but admits exceptions
for the sake of good voice leading \citep[p.~112, rule~5]{Jeppesen1992}. Fux places the upper
voices farther from the bass to allow contrary motion
\citep[p.~75]{Fux1965}: Figs.~104 and 106 reach an outer seventeenth, and
Fig.~104 a fifteenth between $\mathrm L$ and $\mathrm M$ \citep[pp.~79, 81]{Fux1965}.

\item \emph{Diatonic scale.} All six pitches belong to
$\pc^{-1}(\mathrm{Dia})$.\footnote{A tonic in $\mathrm{Dia}$ specifies a mode; modal and cadential
conditions appear in Section~\ref{subsec:prorul3}. The fixed scale omits
context-dependent accidentals, including noncadential ones such as $B\flat$
in Fux's Fig.~105 \citep[p.~80]{Fux1965}.}

\item \emph{Maximum change of a voice.} Each voice moves at most an octave
\citep[p.~109]{Jeppesen1992}:
$|\Delta_{\mathrm L}|,|\Delta_{\mathrm M}|,|\Delta_{\mathrm U}|\le12$.
\end{enumerate}

Up to common translation, rules~1--5 admit 74 simultaneities. Their reduced
interval configurations $\begin{bmatrix}\pc(c)\\\pc(b)\end{bmatrix}$ form exactly the
26-element mask $H_{\mathrm{Fux}}$ of Definition~\ref{def:mask}. This harmonic mask encodes
vertical conditions, not progression or contextual rules.

\subsection{Progression rules}\label{subsec:prorul3}

Classify progressions satisfying the preliminary rules as follows, with
inadmissibility taking precedence.

\Needspace{6\baselineskip}
\begin{itemize}
\item \emph{Inadmissible.} Excluded by at least one of these conditions:
\begin{itemize}
\item[] \emph{Successive triple unisons.} $b=c=b'=c'=0$.
This is the local analogue of the unison-repetition condition
\citep[Section~2]{AriasValeroAgustinAquinoLluisPuebla2023}; positional
restrictions are listed below.

\item[] \emph{Parallel perfect consonances} in any pair
\citep[pp.~22, 76]{Fux1965}; \citep[p.~112, rule~4; p.~287]{Jeppesen1992}:
$k\in\hat K_{\mathrm{perf}}$, $k=k'$, and $\Delta_V\ne0$, with equality of integer
intervals. Fifths have $\pc(k)=7$; octaves and unisons have $\pc(k)=0$.

\item[] \emph{Hidden octaves and unisons between the outer voices.}
For $\mathrm{LU}$: $\pc(k')=0$, $k'\ne k$, and $\Delta_{\mathrm L}\Delta_{\mathrm U}>0$.
This ban is stricter than Jeppesen's advice to avoid outer hidden octaves,
which permits them at cadences and, with care, elsewhere
\citep[p.~176, rule~3; p.~287]{Jeppesen1992}.

\item[] \emph{Hidden fifths between the outer voices with a skip in the
upper voice.} For $\mathrm{LU}$: $\pc(k')=7$, $k'\ne k$,
$\Delta_{\mathrm L}\Delta_{\mathrm U}>0$, and $|\Delta_{\mathrm U}|>2$. The inadmissible/bad boundary
adopts Jeppesen's recommendation of stepwise upper-voice motion
\citep[p.~176, rule~3]{Jeppesen1992}.

\item[] \emph{Melodic tritones} \citep[p.~35]{Fux1965}.
$|\Delta|=6$ for at least one voice.

\item[] \emph{Excessive skips.} $7<|\Delta|<12$ for at least one voice:
sixths and sevenths are excluded in both directions, while the octave
counts as repetition \citep[p.~112, rule~7]{Jeppesen1992}. This condition
excludes the ascending minor sixth, which Fux \citep[p.~27, n.~1]{Fux1965}
and Jeppesen \citep[p.~109]{Jeppesen1992} permit;
Remark~\ref{rem:variants3} gives a variant that admits it.
\end{itemize}

\Needspace{7\baselineskip}
\item \emph{Bad.} Retained but disfavored: not inadmissible, but satisfying
at least one of these conditions:
\begin{itemize}
\item[] \emph{Hidden fifths between the outer voices with a step in the
upper voice.} For $\mathrm{LU}$: $\pc(k')=7$, $k'\ne k$,
$\Delta_{\mathrm L}\Delta_{\mathrm U}>0$, and $0<|\Delta_{\mathrm U}|\le2$. Jeppesen calls these
effective \citep[p.~176]{Jeppesen1992}, whereas Fux accepts the progression
from bar~7 to bar~8 of Fig.~101 only when no alternative is available
\citep[p.~78]{Fux1965}.
Fux's reservation motivates the bad label.

\item[] \emph{Hidden perfect consonances involving the middle voice.}
For $\mathrm{LM}$ or $\mathrm{MU}$: $k'\in\hat K_{\mathrm{perf}}$, $k'\ne k$, and similar motion.
Jeppesen permits them \citep[p.~176, rule~3; p.~287]{Jeppesen1992}.
The bad label reflects Fux's guarded exception for the upper-voice fifth
moving to the final octave in Fig.~100 \citep[pp.~76--77]{Fux1965}.

\item[] \emph{Tritone between the upper voices by similar motion.}
$\pc(c'-b')=6$, $c'-b'\ne c-b$, and $\Delta_{\mathrm M}\Delta_{\mathrm U}>0$.
Fux objects to approaching a diminished fifth directly
\citep[p.~76, Figs.~98--99]{Fux1965}; we apply the same objection to the
enharmonically equivalent augmented fourth.

\item[] \emph{Similar skips in the outer voices.} $c'\in\hat K_{\mathrm{imp}}$,
$\Delta_{\mathrm L}\Delta_{\mathrm U}>0$, $|\Delta_{\mathrm L}|,|\Delta_{\mathrm U}|>2$, and at least one of
$5<|\Delta_{\mathrm L}|<12$ or $5<|\Delta_{\mathrm U}|<12$ holds, while all three voices move
in the same direction. Jeppesen limits the three-part restriction to the
outer voices when all voices move similarly
\citep[p.~112, rule~7; p.~176, rule~2]{Jeppesen1992}; $c'\in\hat K_{\mathrm{imp}}$
retains the weaker two-voice formulation
\citep[Section~2]{AriasValeroAgustinAquinoLluisPuebla2023}.

\item[] \emph{All voices in similar motion.} $\Delta_{\mathrm L}$, $\Delta_{\mathrm M}$, and
$\Delta_{\mathrm U}$ are nonzero and have the same sign. Fux describes Fig.~98's
ascending progression as awkward \citep[p.~76]{Fux1965}; we apply this
indicator in either direction.

\item[] \emph{Ascending parallel sixths.} For some $W\in\{\mathrm M,\mathrm U\}$:
$\pc(k_{\mathrm{L}W}(\xi))\in\{8,9\}$, $\pc(k_{\mathrm{L}W}(\xi'))\in\{8,9\}$,
$\Delta_{\mathrm L}>0$, and $\Delta_W>0$. Fux criticizes ascending sixths on the
downbeat \citep[p.~77]{Fux1965}; we apply this to both bass-containing
pairs, including compound sixths.

\item[] \emph{Unisons between two voices.} Exactly two voices of $\xi'$
share an integer pitch. Fux finds this less harmonious than the octave
\citep[p.~79]{Fux1965}, whereas Jeppesen allows it freely
\citep[p.~176]{Jeppesen1992}.

\item[] \emph{Hidden tritones.} For some pair:
$\pc(w'-v)=6$ or $\pc(w-v')=6$, where $v,w$ and $v',w'$ are its pitches
in the two terms. This extends the two-voice indicator
\citep[Section~2]{AriasValeroAgustinAquinoLluisPuebla2023} to all pairs;
compare Fux's qualified cross-voice tritone rule
\citep[p.~35, n.~9]{Fux1965}.
\end{itemize}

\item \emph{Good.} All remaining progressions.
\end{itemize}

The two-voice-unison test depends only on the simultaneity $\xi'$, so in
a composition it checks each term except the first exactly once. Types
within a category may overlap; similar outer skips are a subtype of all
voices in similar motion.

\par\medskip\noindent\textbf{Rules outside the local classification.}\quad
These contextual checks apply to complete compositions; they neither enter
the counts nor override local labels.
\begin{itemize}
\item \emph{Beginning.} Jeppesen permits a complete tonic triad, with a
major third when all voices enter together, or an ``empty'' triad without a
third \citep[pp.~176--177; p.~288]{Jeppesen1992}. Fux's Figs.~101 and 106
instead begin with an octave and third above the bass
\citep[pp.~78, 81]{Fux1965}.
\item \emph{Ending.} Jeppesen permits a complete tonic triad or omission
of its third or fifth; any third is major
\citep[p.~177; p.~289]{Jeppesen1992}. Fux prefers omitting the third,
requiring a major one if present \citep[p.~77, pp.~80--82]{Fux1965}.
\item \emph{Cadence.} Outside Phrygian, raise the cadential seventh degree
if necessary to a semitone below the tonic \citep[pp.~28--29]{Fux1965};
\citep[pp.~110--111]{Jeppesen1992}. Phrygian requires separate treatment
\citep[pp.~81--82, Fig.~107]{Fux1965}. Jeppesen permits outer hidden
octaves from the penultimate to the final bar \citep[p.~176]{Jeppesen1992}.
\item \emph{Leading tones.} Jeppesen forbids doubling $C\sharp$ in Dorian,
$F$ in Phrygian, $F\sharp$ in Mixolydian, $G\sharp$ in Aeolian, and $B$ in
Ionian \citep[p.~177, rule~6]{Jeppesen1992}. Thus the Phrygian leading
tone is not simply a semitone below the tonic.
\item \emph{Triple unisons.} These are allowed only in the first or last
bar \citep[p.~176, rule~4; p.~288]{Jeppesen1992}, a positional restriction
not enforced by the local repetition condition.
\item \emph{Longer passages.} Avoid more than four parallel thirds or
sixths where possible \citep[p.~112, rule~6; p.~288]{Jeppesen1992}, more
than one consecutive repetition of a note \citep[p.~29, n.~3]{Fux1965},
and two successive hidden perfect consonances \citep[p.~79]{Fux1965}.
Fux encourages variety \citep[p.~75]{Fux1965}.
\item \emph{Preferences.} Fux favors the complete harmonic triad (a
third and a fifth above the bass) unless voice leading warrants otherwise
\citep[pp.~71--72]{Fux1965}; \citep[p.~175]{Jeppesen1992}. Contrary
motion is preferred \citep[p.~27]{Fux1965};
\citep[p.~112, rule~8]{Jeppesen1992}. Above \emph{mi} in the bass, the
choice of fifth or sixth depends on the following bass note
\citep[pp.~73--75]{Fux1965}.
\item \emph{Absolute ranges.} Fux recommends staying within the five
lines of each clef unless there is a serious reason to exceed them
\citep[p.~79 and n.~7]{Fux1965}; Jeppesen gives individual voice ranges
\citep[p.~110]{Jeppesen1992}.
\end{itemize}

\subsection{Strict style modulo translation table}\label{subsec:strtrans3}

Let $\mathcal P_{\mathrm{Dia}}$ be the set of progressions satisfying the preliminary
rules in the fixed scale $\mathrm{Dia}$.
Put $d=l'-l$ and subtract the initial lower pitch from all six pitches:
\[
\operatorname{nor}(\xi,\xi')=\bigl(bx_1+cx_2,\ d+b'x_1+c'x_2\bigr).
\]
The finite set $\mathcal P=\operatorname{nor}(\mathcal P_{\mathrm{Dia}})$ has one representative per
common-translation class realizable in $\mathrm{Dia}$. Its normalized pitches need
not lie in $\pc^{-1}(\mathrm{Dia})$; instead, they satisfy
\begin{equation}\label{eq:common-scale3}
\exists a\in\Z_{12}:\quad
a+\pc\bigl(\{0,b,c,d,d+b',d+c'\}\bigr)\subseteq\mathrm{Dia},
\end{equation}
with $\pc$ acting elementwise. The other preliminary conditions and the
interval-based progression predicates are preserved. The motions become
\[
(\Delta_{\mathrm L},\Delta_{\mathrm M},\Delta_{\mathrm U})=(d,\ d+b'-b,\ d+c'-c).
\]
Write $\mathcal P=\mathcal I\sqcup\mathcal B\sqcup\mathcal G$ for the
induced inadmissible/bad/good partition.

Tables~\ref{tab:prog3} and~\ref{tab:prog3pairs} give the enumeration by
type and voice pair. The accompanying Python script also reproduces the
two-voice counts (1\,057 progressions: 671 inadmissible, 64 bad, and 322
good) and their reduced and refined counts
\citep[Tables~1--2]{AriasValeroAgustinAquinoLluisPuebla2023}. This checks numerical
agreement with the two-voice reference, not the historical validity of the
three-voice classification.

\begin{table}[htbp]
\centering
\small
\begin{tabularx}{\textwidth}{|>{\raggedright\arraybackslash}p{0.23\textwidth}|>{\raggedright\arraybackslash}X|r|}
\hline
Category & Type & Progressions\\
\hline
Inadmissible (30\,976) & Successive triple unisons & 25\\
 & Parallel fifths & 1\,556\\
 & Parallel octaves and unisons & 6\,766\\
 & Hidden octaves and unisons (outer voices) & 4\,293\\
 & Hidden fifths, upper voice skipping (outer voices) & 2\,468\\
 & Melodic tritones & 7\,468\\
 & Excessive skips & 21\,576\\
\hline
Bad (5\,581) & Hidden fifths, upper voice stepping (outer voices) & 163\\
 & Hidden perfect consonances involving the middle voice & 2\,678\\
 & Tritone between the upper voices by similar motion & 207\\
 & Similar skips in the outer voices & 411\\
 & All voices in similar motion & 1\,628\\
 & Ascending parallel sixths & 355\\
 & Unisons between two voices & 2\,456\\
 & Hidden tritones & 1\,134\\
\hline
Good & & 4\,423\\
\hline
Total & & 40\,980\\
\hline
\end{tabularx}
\caption{Three-voice strict style progressions modulo translation, by type. A progression may belong to several types. The three main categories are disjoint; bad types are counted only
among progressions that are not inadmissible.}
\label{tab:prog3}
\end{table}

\begin{table}[htbp]
\centering
\small
\begin{tabularx}{\textwidth}{|>{\raggedright\arraybackslash}X|r|r|r|}
\hline
Type & $\mathrm{LM}$ & $\mathrm{LU}$ & $\mathrm{MU}$\\
\hline
Parallel fifths & 558 & 770 & 338\\
Parallel octaves and unisons & 2\,126 & 1\,698 & 3\,182\\
Hidden octaves and unisons & -- & 4\,293 & --\\
Hidden fifths, upper voice skipping & -- & 2\,468 & --\\
Hidden fifths, upper voice stepping & -- & 163 & --\\
Hidden perfect consonances & 1\,271 & -- & 1\,430\\
Hidden tritones & 412 & 387 & 520\\
\hline
\end{tabularx}
\caption{Pair-dependent types of Table~\ref{tab:prog3} by pair of voices. A
dash marks a type that is not defined for that pair. The counting convention
is the same as in Table~\ref{tab:prog3}.}
\label{tab:prog3pairs}
\end{table}

\FloatBarrier
\begin{remark}\label{rem:variants3}
Each variant changes only the stated choice. Counts are ordered as
inadmissible, bad, and good.
\begin{itemize}
\item Allowing the ascending minor sixth replaces the excessive-skip
condition by $7<|\Delta|<12$ and $\Delta\ne8$, giving
(29\,845 / 6\,266 / 4\,869).
\item Reclassifying outer hidden octaves and outer hidden fifths with an
upper-voice skip as bad, without a cadential test, gives
(29\,573 / 6\,984 / 4\,423).
\item Replacing the upper-voice distance bound by $c-b\le12$ leaves 66
simultaneities and 33\,826 progressions (25\,426 / 4\,795 / 3\,605).
\item Replacing the outer-voice distance bound by $c\le24$ leaves 61
simultaneities and 31\,171 progressions (23\,603 / 4\,356 / 3\,212).
\end{itemize}
Fux rejects skips from a more remote consonance into an octave or unison,
in writing for more than two voices as well \citep[pp.~38--39, Figs.~19--20]{Fux1965}.
His examples suggest the additional condition, for some pair $VW$,
\[
\pc(k')=0,\qquad k>k',\qquad
\Delta_V\Delta_W<0,\qquad |\Delta_W|>2.
\]
Including compound octaves, this holds for 4\,578 progressions. Adding it
to the inadmissible types reclassifies 1\,164 previously bad or good
progressions, giving (32\,140 / 4\,707 / 4\,133). This condition is not used
in the main classification.
\end{remark}

\par\medskip\noindent\textbf{The reduced three-voice strict style.}\quad
For comparison with a pitch-class model, reduce normalized progressions
coefficientwise modulo 12
\citep[Section~5]{AriasValeroAgustinAquinoLluisPuebla2023}:
\[
q:\mathcal P\longrightarrow
\bigl(\Z_{12}[\varepsilon_1,\varepsilon_2]\bigr)^2,
\qquad q(\xi,\xi')=(\rho(\xi),\rho(\xi')).
\]
Set $\overline{\mathcal P}=q(\mathcal P)$. Each nonempty fiber of
integer-pitch realizations determines a label, with good taking precedence
over bad, and bad over inadmissible:
\begin{align*}
\overline{\mathcal G}&=q(\mathcal G),\\
\overline{\mathcal B}&=q(\mathcal B)\setminus q(\mathcal G),\\
\overline{\mathcal I}&=\overline{\mathcal P}\setminus
                     \bigl(q(\mathcal G)\cup q(\mathcal B)\bigr).
\end{align*}
The 4\,159 reduced progressions comprise 1\,415 inadmissible, 973 bad,
and 1\,771 good progressions. Refine the good class by
\begin{align*}
\overline{\mathcal G}_{\mathrm{gg}}
 &=q(\mathcal G)\setminus\bigl(q(\mathcal I)\cup q(\mathcal B)\bigr),\\
\overline{\mathcal G}_{\mathrm{amb}}
 &=q(\mathcal G)\cap q(\mathcal I),\\
\overline{\mathcal G}_{\mathrm{gb}}
 &=\bigl(q(\mathcal G)\cap q(\mathcal B)\bigr)\setminus q(\mathcal I).
\end{align*}
These \emph{good-good}, \emph{ambiguous}, and \emph{good-bad} classes have
2, 1\,706, and 63 elements. An ambiguous fiber contains good and
inadmissible realizations, so a reduced good label does not certify every
registral realization.

With the wider registral bounds, even a locally good progression such as
\[
(\mathrm L,\mathrm M,\mathrm U)=(0,4,19)\longrightarrow(7,11,14)
\]
can project to consecutive fifths: its motions are $(7,7,-5)$, and the
outer interval changes from 19 to 7 by contrary motion; both intervals reduce to 7.
Forbidding every reduced transition from 7 to 7 therefore excludes some
good integer-pitch realizations. Comparison with an algebraic successor
relation must use the fiber-based classification, rather than applying the
integer parallel-fifths predicate directly to pitch classes.

\section{Evaluation of the three-voice model}\label{subsec:evaluation}

Following \citet[Sections~5--6]{AriasValeroAgustinAquinoLluisPuebla2023}, we compare two
successor relations with the reduced strict style of Section~\ref{sec:strict3}. The first,
V0, is the relation $S_{H_{\mathrm{Fux}}}$ of Sections~\ref{sec:construction}
and~\ref{sec:fuxian}. The second, V1, is a
baseline that omits the joint maximization:
\[
S_{H_{\mathrm{Fux}}}^{\mathrm{V1}}(\xi)=\bigcup_{\mathbf g\in\Sigma(\xi)}C_{H_{\mathrm{Fux}}}(\xi;\mathbf g),
\]
that is, $\Gamma_{H_{\mathrm{Fux}}}(\xi)$ is replaced by $\Sigma(\xi)$ in~\eqref{eq:successors}. Since
$\Sigma(\xi)$ is a product, V1 applies the two-voice model independently to each active
pair of voices, and $S_{H_{\mathrm{Fux}}}(\xi)\subseteq S_{H_{\mathrm{Fux}}}^{\mathrm{V1}}(\xi)$. With the source lower voice
normalized to pitch class~0, a reduced progression $(\rho(\xi),\rho(\xi'))$ is \emph{allowed} if $\rho(\xi')$ is an
admitted successor of $\rho(\xi)$ and \emph{forbidden} otherwise; a repetition
$\rho(\xi')=\rho(\xi)$ is non-polarized and left undecided. The same code reproduces
\citet[Table~3]{AriasValeroAgustinAquinoLluisPuebla2023} and
Table~\ref{tab:successor-counts}.

Only the melodic-tritone rule keeps its generality under the reduction, so
$\overline{\mathcal I}^{*}$ consists of the projected melodic tritones; parallel fifths lose
it because a compound fifth approached by contrary motion has the same reduction
(Section~\ref{sec:strict3}). As the good class is almost entirely ambiguous, the refined
semantics is not informative.

\begin{table}[htbp]
\centering
\small
\begin{tabular}{llrrrrrr}
\toprule
Variant & & $\overline{\mathcal I}$ (1\,415) & $\overline{\mathcal I}^{*}$ (751) & $\overline{\mathcal B}$ (973) & $\overline{\mathcal G}$ (1\,771) & gb (63) & amb (1\,706)\\
\midrule
V0 & allowed & 513 & 268 & 555 & 954 & 22 & 931\\
 & forbidden & 902 & 483 & 416 & 793 & 39 & 753\\
\midrule
V1 & allowed & 794 & 390 & 666 & 1\,198 & 29 & 1\,168\\
 & forbidden & 621 & 361 & 305 & 549 & 32 & 516\\
\bottomrule
\end{tabular}
\caption{Model decisions on the reduced strict style: inadmissible, inadmissible with preserved generality, bad, and good progressions, the last refined into good-bad (gb) and ambiguous (amb); totals in parentheses. Not shown, and identical for both variants: the undecided repetitions (2 bad, 24 good, of which 2 gb and 22 amb) and the 2 good-good progressions (1 allowed, 1 forbidden).}
\label{tab:eval-comparison}
\end{table}

\begin{table}[htbp]
\centering
\small
\begin{tabular}{lrrrrrrrr}
\toprule
 & \multicolumn{3}{c}{Forbidden (\%)} & & & & & Integer (\%)\\
\cmidrule(lr){2-4}
Variant & $\overline{\mathcal I}$ & $\overline{\mathcal B}$ & $\overline{\mathcal G}$ & OR & Matches & Mismatches & Chance & $\mathcal I$ / $\mathcal G$\\
\midrule
V0 & 63.7 & 42.8 & 45.4 & 2.12 & 2\,272 & 1\,306 & 2\,073 & 52.6 / 50.1\\
V1 & 43.9 & 31.4 & 31.4 & 1.71 & 2\,124 & 1\,343 & 1\,975 & 36.2 / 37.9\\
\bottomrule
\end{tabular}
\caption{Summary. Percentages exclude repetitions; OR is the odds ratio of being forbidden, inadmissible versus good. Matches and mismatches as in \citet[Section~6]{AriasValeroAgustinAquinoLluisPuebla2023}; Chance is their expectation for a model forbidding the same proportion at random. The last column gives the forbidden percentages for the inadmissible and good progressions at the integer level (40\,980 progressions in total).}
\label{tab:eval-summary}
\end{table}

\begin{table}[htbp]
\centering
\small
\begin{tabular}{lrr}
\toprule
Type & V0 & V1\\
\midrule
\multicolumn{3}{l}{\emph{Inadmissible types, within $\overline{\mathcal I}$}}\\
Parallel fifths & 214/214 & 214/214\\
\quad $\mathrm{LM}$ & 88/88 & 88/88\\
\quad $\mathrm{LU}$ & 90/90 & 90/90\\
\quad $\mathrm{MU}$ & 56/56 & 56/56\\
Parallel octaves and unisons & 345/597 & 172/597\\
\quad $\mathrm{LM}$ & 116/207 & 48/207\\
\quad $\mathrm{LU}$ & 128/205 & 78/205\\
\quad $\mathrm{MU}$ & 111/207 & 48/207\\
Hidden octaves and unisons (outer) & 148/240 & 109/240\\
Hidden fifths, upper voice skipping (outer) & 105/199 & 49/199\\
Melodic tritones & 483/751 & 361/751\\
Excessive skips & 816/1\,290 & 542/1\,290\\
Successive triple unisons & 5/11 & 1/11\\
\midrule
\multicolumn{3}{l}{\emph{Bad types, within $\overline{\mathcal B}$}}\\
Hidden fifths, upper voice stepping (outer) & 12/57 & 10/57\\
Hidden perfect consonances (middle voice) & 152/513 & 93/513\\
Tritone between upper voices, similar motion & 47/68 & 46/68\\
Similar skips in the outer voices & 17/47 & 12/47\\
All voices in similar motion & 145/296 & 102/296\\
Ascending parallel sixths & 75/112 & 43/112\\
Unisons between two voices & 121/385 & 78/385\\
Hidden tritones & 252/451 & 200/451\\
\bottomrule
\end{tabular}
\caption{Forbidden/decided reduced progressions per type; a progression counts for a type if its fiber contains a realization of that type; pair counts may overlap.}
\label{tab:eval-rules}
\end{table}

\paragraph{Agreement.} Both variants forbid inadmissible progressions more often than good
ones and exceed the chance level of matches (Table~\ref{tab:eval-summary}). All inadmissible reduced progressions with
parallel fifths are forbidden, in each pair and in both variants
(Proposition~\ref{prop:fifths}). Relative to V1, the joint maximization of V0 forbids 281
additional inadmissible and 244 additional good progressions (Table~\ref{tab:eval-comparison}). It thereby raises the odds ratio
from 1.71 to 2.12 and doubles the number of detected parallel octaves and unisons (345 against
172 of 597; Table~\ref{tab:eval-rules}), which nevertheless remain allowed in 252 cases.

\paragraph{Limits.} At the integer level, V0 hardly separates inadmissible from good
progressions (52.6\% against 50.1\% forbidden), partly because 12\,183 of the 30\,976
inadmissible ones violate only registral conditions (excessive skips, triple unisons) that
a pitch-class model cannot express. Most other mismatches stem from the two-voice model:
549 of the 793 good progressions forbidden by V0 are also forbidden by V1. They include the
186 progressions that keep an active pair of voices stationary, which, by~\eqref{eq:stationary}, neither
variant admits.

\paragraph{Robustness.} Under the five alternative choices of Remark~\ref{rem:variants3},
matches change by at most 84 and forbidden percentages by at most 3.2 points; V0 has more
matches than V1 in every case.

\section{Fuxian counterpoint and neo-Riemannian theory}\label{sec:noll}

Mazzola relates harmony to counterpoint through the Riemann dichotomy
\citep[Section~30.2.1]{Mazzola2017}; see also
\citet[Section~3.5]{AgustinAquinoJunodMazzola2015}. The eight affine
maps that send the dominant triad into the tonic triad generate a
monoid of 72 of the 144 affine maps of \(\Z_{12}\), and a bijection
from the affine maps onto \(\Z_{12}[\varepsilon]\) carries this monoid
onto the consonant intervals \(K[\varepsilon]\). The correspondence is
between sets and does not involve admitted successors; Mazzola leaves
its contrapuntal consequences open \citep[p.~523]{Mazzola2017}.

Subsection~\ref{subsec:noll-pairs} determines which other pairs of
triads generate the same monoid, and Subsection~\ref{subsec:noll-line}
obtains Mazzola's bijection as the lower--middle projection of an
embedding into the three-voice state space.
Subsections~\ref{subsec:transfer} and~\ref{subsec:parsimonious} examine
which realizations of transporters and of the neo-Riemannian
transformations \(\mathsf P\), \(\mathsf L\), and \(\mathsf R\) the
model admits.

\subsection{Noll pairs and the Riemann dichotomy}\label{subsec:noll-pairs}

Work in semitone coordinates on \(\Z_{12}\), with all arithmetic
modulo~12. Let
\[
 \mathcal A=\{T_u\cdot v:u,v\in\Z_{12}\},\qquad (T_u\cdot v)(z)=u+vz,
\]
be the monoid of affine maps under composition. Its group of units is
\(G=\AffGL(\Z_{12})\), consisting of the maps with
\(v\in\Z_{12}^{\times}=\{1,5,7,11\}\).
Write \(\lambda(T_u\cdot v)=v\) for the linear part and
\(\mathcal U=\{T_t\cdot1:t\in\Z_{12}\}\) for the translation group.
For the Fuxian consonance set \(K=\{0,3,4,7,8,9\}\), the set
\(K_7=7K=\{0,1,3,4,8,9\}\) is multiplicatively closed and contains~1.
Thus \(\mathcal R=\lambda^{-1}(K_7)\) is a submonoid of \(\mathcal A\) with
\[
 |\mathcal R|=72,\qquad \mathcal R\cap G=\mathcal U,
\]
and the \emph{Riemann dichotomy}
\citep[Section~3.5]{AgustinAquinoJunodMazzola2015} is the partition
\((\mathcal R/\mathcal A\setminus\mathcal R)\) of the 144 affine maps.
The cited reference uses fifth coordinates: the change \(z\mapsto7z\)
sends \(T_u\cdot v\) to \(T_{7u}\cdot v\), leaving linear parts unchanged.

For \(r\in\Z_{12}\), define the major and minor triads
\[
 X_r^+=\{r,r+4,r+7\},\qquad X_r^-=\{r,r+3,r+7\},\qquad
 \mathcal T=\{X_r^\sigma:r\in\Z_{12},\ \sigma\in\{+,-\}\}.
\]
Here \(r\) is the root and \(\sigma\) the mode. From now on, \(X\) and \(Y\)
denote triads, not the consonance and dissonance sets of
Section~\ref{sec:construction}. For \(X,Y\in\mathcal T\),
define the \emph{transporter} and the \emph{endomorphism monoid} by
\[
 \operatorname{Tr}(X,Y)=\{f\in\mathcal A:f(X)\subseteq Y\},\qquad
 \operatorname{End}(X)=\operatorname{Tr}(X,X).
\]
For \(\mathcal F\subseteq\mathcal A\), let \(\langle\mathcal F\rangle\) denote the submonoid
generated by \(\mathcal F\), including the identity \(\mathrm{id}=T_0\cdot1\).
We call an ordered pair \((X,Y)\) of distinct triads a \emph{Noll pair}
if \(\langle\operatorname{Tr}(X,Y)\rangle=\mathcal R\).

\begin{theorem}[Classification of Noll pairs]\label{thm:noll-pairs}
Let \(X=X_r^\sigma\ne Y=X_s^\tau\), put \(d=s-r\), and set
\(\mathcal M=\langle\operatorname{Tr}(X,Y)\rangle\). Then
\[
 \mathcal M\cap G=
 \begin{cases}
 \{T_{nd}\cdot1:n\in\Z_{12}\},&\sigma=\tau,\\
 \{\mathrm{id},T_{r+s+7}\cdot11\},&\sigma\ne\tau.
 \end{cases}
\]
Moreover, \(\mathcal M\subseteq\mathcal R\) if and only if \(\sigma=\tau\), and
\[
 (X,Y)\text{ is a Noll pair}
 \quad\Longleftrightarrow\quad
 \sigma=\tau\ \text{ and }\ d\in\{1,5,7,11\}.
\]
Thus Noll pairs have the same mode and roots a semitone or a perfect
fifth apart, in either direction. There are 96 ordered Noll pairs,
or 48 unordered pairs.
\end{theorem}

\begin{proof}
For \(C=X_0^+=\{0,4,7\}\), we have
\citep{Noll2005,FioreNoll2011}
\[
 \operatorname{End}(C)=
 \{T_0\cdot1,T_0\cdot0,T_4\cdot0,T_7\cdot0,
   T_7\cdot3,T_0\cdot4,T_4\cdot8,T_4\cdot9\}.
\]
Indeed, \(a=f(0)\) and \(b=f(7)\) determine
\(f=T_a\cdot\bigl(7(b-a)\bigr)\); among the nine choices \((a,b)\in C^2\),
the condition \(f(4)\in C\) excludes only \((7,0)\).
Every triad is \(\psi(C)\) for a transposition or inversion \(\psi\in G\),
and \(\operatorname{End}(\psi(C))=\psi\circ\operatorname{End}(C)\circ\psi^{-1}\).
Conjugation preserves linear parts, so for every \(X\in\mathcal T\),
\begin{equation}\label{eq:noll-end}
 \operatorname{End}(X)\cap G=\{\mathrm{id}\},\qquad
 \lambda(\operatorname{End}(X))=K_7,\qquad
 \mathcal U\circ\operatorname{End}(X)=\mathcal R.
\end{equation}
Here composition of sets means all pairwise composites; the last
identity follows because left translation allows every constant term.

The map
\[
 \theta=\begin{cases}
 T_d\cdot1,&\sigma=\tau,\\
 T_{r+s+7}\cdot11,&\sigma\ne\tau
 \end{cases}
\]
satisfies \(\theta(X)=Y\), hence
\(\operatorname{Tr}(X,Y)=\theta\circ\operatorname{End}(X)\).
By \eqref{eq:noll-end}, \(\theta\) is the transporter's only invertible element.
Linear parts multiply under composition, and a product in \(\Z_{12}\)
is a unit exactly when each factor is a unit. Therefore
\(\mathcal M\cap G=\{\theta^n:n\ge0\}\), giving the asserted formula since
\(\theta^2=\mathrm{id}\) when the modes differ.
In that case \(\theta\notin\mathcal R\); when the modes agree,
\(\theta\circ\operatorname{End}(X)\subseteq\mathcal R\), so
\(\mathcal M\subseteq\mathcal R\).

If \(\mathcal M=\mathcal R\), its invertible elements are all translations,
which forces equal modes and \(d\in\Z_{12}^{\times}\).
Conversely, these conditions imply \(\mathcal U\subseteq\mathcal M\) and
\(\operatorname{End}(X)=\theta^{-1}\circ\operatorname{Tr}(X,Y)\subseteq\mathcal M\).
Equation~\eqref{eq:noll-end} then gives
\(\mathcal R\subseteq\mathcal M\subseteq\mathcal R\).
Finally, each of the 24 triads has four partners, and reversing a pair
preserves the criterion since \(-\Z_{12}^{\times}=\Z_{12}^{\times}\).
\end{proof}

For example, among the six major and minor triads of C major, the
unordered Noll pairs are exactly
\(\{X_0^+,X_5^+\}\), \(\{X_0^+,X_7^+\}\),
\(\{X_2^-,X_9^-\}\), and \(\{X_4^-,X_9^-\}\);
the corresponding statement for any major key follows by transposition.
Thus the dominant--tonic pair \((X_7^+,X_0^+)\) generates
\(\mathcal R\), but so do its reverse \((X_0^+,X_7^+)\), the
subdominant--tonic pair \((X_5^+,X_0^+)\), and the same-mode semitone pairs. No mode-changing pair generates
\(\mathcal R\).

The neo-Riemannian transformations \(\mathsf P\), \(\mathsf L\), and
\(\mathsf R\) send a major or minor triad to the triad of the other mode
that shares with it the perfect fifth, the minor third, or the major
third, respectively \citep{CransFioreSatyendra2009}; for
instance, \(\mathsf P(X_r^+)=X_r^-\), \(\mathsf L(X_r^+)=X_{r+4}^-\),
and \(\mathsf R(X_r^+)=X_{r+9}^-\).

\begin{remark}\label{rem:nr-pairs}
For \(Y\in\{\mathsf P(X),\mathsf L(X),\mathsf R(X)\}\), the inversion
\(T_{r+s+7}\cdot11\) in the proof of Theorem~\ref{thm:noll-pairs} is the
map by which the transformation acts on \(X\); for \(X=X_0^+\) it is
\(T_7\cdot11\), \(T_{11}\cdot11\), and \(T_4\cdot11\), respectively.
Since these pairs change mode, Theorem~\ref{thm:noll-pairs} shows that
replacing the dominant--tonic pair by another pair of triads cannot
extend the correspondence of \citet[Section~30.2.1]{Mazzola2017} to
\(\mathsf P\), \(\mathsf L\), and \(\mathsf R\).
\end{remark}

\subsection{The Riemann dichotomy in three voices}\label{subsec:noll-line}

Extend the action of \(\mathcal A\) to
\(\Z_{12}[\varepsilon_1,\varepsilon_2]\) by applying each map to all three voices:
\[
 (T_u\cdot v)\cdot\begin{bmatrix}c\\b\\a\end{bmatrix}
 =\begin{bmatrix}vc\\vb\\u+va\end{bmatrix}.
\]
For \(X=X_r^\sigma\), set
\(m_1=r\), \(m_3=r+t_\sigma\), and \(m_5=r+7\), where
\(t_+=4\) and \(t_-=3\). A \emph{voicing type}
\(\kappa=\kappa_1\kappa_2\kappa_3\in\mathcal V:=\{135,153,351,315\}\)
lists the chord members in the lower, middle, and upper voices and defines
\[
 \xi_\kappa(X)=
 \begin{bmatrix}
 m_{\kappa_3}-m_{\kappa_1}\\
 m_{\kappa_2}-m_{\kappa_1}\\
 m_{\kappa_1}
 \end{bmatrix}.
\]
Of the six voicings of \(X\) that use all three chord members, these are
exactly the four that lie in \(K_{H_{\mathrm{Fux}}}\); the other two place
a perfect fourth above the bass.

\begin{proposition}\label{prop:noll-line}
Let \(\xi^{\circ}=\xi_{153}(X_0^+)=\begin{bmatrix}4\\7\\0\end{bmatrix}\). The map
\[
 \operatorname{ev}\colon\mathcal A\to\Z_{12}[\varepsilon_1,\varepsilon_2],\qquad
 \operatorname{ev}(T_u\cdot v)=(T_u\cdot v)\cdot\xi^{\circ}
 =\begin{bmatrix}4v\\7v\\u\end{bmatrix},
\]
is injective. The composite
\[
 \varphi=\pi_{\mathrm{LM}}\circ\operatorname{ev},\qquad
 \varphi(T_u\cdot v)=u+7v\,\varepsilon,
\]
is the bijection \(\mathcal A\to\Z_{12}[\varepsilon]\) of
\citet[Section~30.2.1]{Mazzola2017} in semitone coordinates, and
\(\varphi(\mathcal R)=K[\varepsilon]\). The image
\(\operatorname{ev}(\mathcal R)\) consists of the 72 states of
\(K_{H_{\mathrm{Fux}}}\) with configurations
\(h_1,h_7,h_{13},h_{18},h_{22},h_{23}\).
\end{proposition}

\begin{proof}
The lower-voice coordinate of \(\operatorname{ev}(T_u\cdot v)\) gives
\(u\), and the middle coordinate \(7v\) gives \(v\) because \(7\cdot7=1\)
in \(\Z_{12}\); hence \(\operatorname{ev}\) is injective. In fifth
coordinates the bijection is \(T_t\cdot y\mapsto t+y\varepsilon\). The
change \(z\mapsto7z\) sends \(T_u\cdot v\) to \(T_{7u}\cdot v\) and a
dual number \(a+b\varepsilon\) to \(7a+7b\varepsilon\), so in semitone
coordinates the bijection is
\(T_u\cdot v\mapsto u+7v\varepsilon=\pi_{\mathrm{LM}}(\operatorname{ev}(T_u\cdot v))\).
Since \(7K_7=K\), we get
\(\varphi(\mathcal R)=\Z_{12}+\varepsilon K=K[\varepsilon]\). Finally,
\(\begin{bmatrix}4v\\7v\end{bmatrix}\) is \(h_1,h_{18},h_{23},h_{13},h_{22},h_7\) for
\(v=0,1,3,4,8,9\), respectively.
\end{proof}

\begin{remark}\label{rem:noll-line-edges}
By~\eqref{eq:projection}, the bijection
\(\varphi\circ\operatorname{ev}^{-1}\colon\operatorname{ev}(\mathcal R)\to K[\varepsilon]\),
which is the restriction of \(\pi_{\mathrm{LM}}\), identifies the subgraph of
\((K_{H_{\mathrm{Fux}}},E_{H_{\mathrm{Fux}}})\) induced on
\(\operatorname{ev}(\mathcal R)\) with a spanning subgraph of the
two-voice graph on \(K[\varepsilon]\), whose edges are the pairs
\((\zeta,\zeta')\) with \(\zeta'\in S_K(\zeta)\). The inclusion is
proper. For a fixed source lower-voice pitch class,
the supplementary material gives 244 admitted targets in
\(\operatorname{ev}(\mathcal R)\) and Table~\ref{tab:fuxian-forbidden}
gives 382, each summed over the six sources; hence the two graphs have
2\,928 and 4\,584 edges. Whether a variant of the joint maximization
recovers every two-voice transition on \(\operatorname{ev}(\mathcal R)\)
is open.
\end{remark}

\begin{remark}\label{rem:doublings}
For a major or minor triad \(X\), a type \(\kappa\in\mathcal V\), and
\(f=T_u\cdot v\in\operatorname{End}(X)\), the state
\(f\cdot\xi_\kappa(X)\) has configuration \(v\begin{bmatrix}c\\b\end{bmatrix}\), where
\(\begin{bmatrix}c\\b\end{bmatrix}\) is the configuration of \(\xi_\kappa(X)\). As \(f\) ranges
over the non-invertible endomorphisms and \((X,\kappa)\) over all
triads and types, these configurations are exactly the 13 doubled
configurations \(\begin{bmatrix}t\\0\end{bmatrix}\), \(\begin{bmatrix}0\\t\end{bmatrix}\), \(\begin{bmatrix}t\\t\end{bmatrix}\) with
\(t\in K\setminus\{7\}\). Together with the eight complete voicings of
major and minor triads, they make up 21 of the 26 elements of
\(H_{\mathrm{Fux}}\); the other five are the empty fifths \(h_4\),
\(h_{16}\), \(h_{19}\), which no endomorphism produces, and the
first-inversion diminished triads \(h_{11}\) and \(h_{24}\). Thus, apart
from the empty fifths, every doubled sonority of the mask is a complete
triad with two members collapsed by an endomorphism.
\end{remark}

\subsection{Contrapuntal realizations of transporters}\label{subsec:transfer}

For same-mode triads \(X=X_r^\sigma\) and \(Y=X_s^\sigma\), put
\(d=s-r\). Translation gives the
bijection
\[
 \operatorname{End}(Y)\longrightarrow\operatorname{Tr}(X,Y),\qquad
 e\longmapsto f=e\circ(T_d\cdot1).
\]
The \emph{realization} of \(f\) in type \(\kappa\) is the transition
\begin{equation}\label{eq:realization}
 \xi_\kappa(X)\longrightarrow f\cdot\xi_\kappa(X)
 =e\cdot\xi_\kappa(Y).
\end{equation}
Let
\[
 \operatorname{Adm}(X,Y)=
 \bigl\{(f,\kappa)\in\operatorname{Tr}(X,Y)\times\mathcal V:
 f\cdot\xi_\kappa(X)\in S_{H_{\mathrm{Fux}}}(\xi_\kappa(X))\bigr\}.
\]
There are \(8\cdot4=32\) transporter--type pairs before imposing
admissibility; coincident target states are not identified.

\begin{theorem}[Transfer principle]\label{thm:transfer}
Fix \(Y\in\mathcal T\), \(e\in\operatorname{End}(Y)\), and
\(\kappa\in\mathcal V\). For \(X\ne Y\) of the same mode as \(Y\),
admissibility of~\eqref{eq:realization} depends on \(X\) only through
\(\delta=\gcd(d,6)\). The gcd is independent of the integer
representative chosen. The counts are
\[
\begin{array}{c|rrrr}
\delta&1&2&3&6\\\hline
|\operatorname{Adm}(X,Y)|\text{, major}&22&16&16&10\\
|\operatorname{Adm}(X,Y)|\text{, minor}&22&12&16&7
\end{array}
\]
and hence:
\begin{enumerate}[label=(\alph*)]
\item If \((X,Y)\) and \((X',Y)\) are Noll pairs, with \(X'=X_{r'}^\sigma\), then
\[
 (f,\kappa)\longmapsto
 \bigl(f\circ(T_{r-r'}\cdot1),\kappa\bigr)
\]
restricts to a bijection
\(\operatorname{Adm}(X,Y)\to\operatorname{Adm}(X',Y)\).
\item Among ordered pairs of distinct same-mode triads,
\(|\operatorname{Adm}(X,Y)|\) is maximal, with value~22, exactly
at the Noll pairs.
\end{enumerate}
\end{theorem}

\begin{proof}
Simultaneous translation of all voices by \(t\in\Z_{12}\) commutes with the pairwise
projections, preserves \(K[\varepsilon]\) and \(D[\varepsilon]\), and
conjugates \(p_x[\varepsilon]\) to \(p_{x+t}[\varepsilon]\).
It therefore preserves both maximization stages and the successor
relation. Conjugating \(e\) by the same translation, we may normalize
\(Y=X_0^\sigma\).

Write \(\xi_\kappa(Y)=\begin{bmatrix}c\\b\\a\end{bmatrix}\) and \(e=T_u\cdot v\).
The source and target configurations in~\eqref{eq:realization} are
\(\begin{bmatrix}c\\b\end{bmatrix}\) and \(\begin{bmatrix}vc\\vb\end{bmatrix}\), and the lower-voice displacement is
\[
 j=u+(v-1)a+d.
\]
Substitution of the eight endomorphisms in Table~\ref{tab:transfer}
shows that all target configurations belong to \(H_{\mathrm{Fux}}\).
The admissibility of each realization can therefore be read off from the
supplementary material.
Its tables are periodic in \(j\) modulo~6; for each of the
\(2\cdot8\cdot4=64\) choices \((\sigma,e,\kappa)\), substitution
also gives equal decisions at \(d\equiv1,5\) and at
\(d\equiv2,4\pmod6\). Table~\ref{tab:transfer} records the resulting
admitted values of \(\gcd(d,6)\); counting entries containing each
value gives the stated totals.

For (a), the displayed correspondence keeps \(e\) and \(\kappa\)
fixed. Both root differences have gcd~1 with~6 by
Theorem~\ref{thm:noll-pairs}. The same theorem identifies \(\delta=1\)
exactly with Noll pairs, proving (b).
\end{proof}

\begin{table}[htbp]
\centering
\small
\setlength{\tabcolsep}{4pt}
\begin{tabular}{lcccc@{\qquad}lcccc}
\toprule
\multicolumn{5}{c}{\(Y=X_0^+\)} & \multicolumn{5}{c}{\(Y=X_0^-\)}\\
\cmidrule(lr){1-5}\cmidrule(lr){6-10}
\(e\)&135&153&351&315&\(e\)&135&153&351&315\\
\midrule
\(T_0\cdot1\)&\(\varnothing\)&\(\varnothing\)&\(\varnothing\)&\(\varnothing\)&
\(T_0\cdot1\)&\(\varnothing\)&\(\varnothing\)&1&\(\varnothing\)\\
\(T_0\cdot0\)&1,3&1,3&\(\ast\)&\(\ast\)&
\(T_0\cdot0\)&2,6&\(\ast\)&2,6&2,6\\
\(T_4\cdot0\)&1,3&1,3&\(\ast\)&\(\ast\)&
\(T_3\cdot0\)&1,3&\(\ast\)&1,3&1,3\\
\(T_7\cdot0\)&2,6&2,6&\(\ast\)&\(\ast\)&
\(T_7\cdot0\)&1,3&\(\ast\)&1,3&1,3\\
\(T_7\cdot3\)&2&2&1,2&1,2&
\(T_3\cdot3\)&2&1,2&2&2\\
\(T_0\cdot4\)&1,3&1,3&\(\varnothing\)&\(\varnothing\)&
\(T_3\cdot4\)&1,3&\(\varnothing\)&1,3&1,3\\
\(T_4\cdot8\)&1,3&1,3&\(\ast\)&\(\ast\)&
\(T_7\cdot8\)&1,3&\(\ast\)&1,3&1,3\\
\(T_4\cdot9\)&1&1&1,2&1,2&
\(T_0\cdot9\)&1&1,2&1&1\\
\bottomrule
\end{tabular}
\caption{Admitted values of \(\delta=\gcd(d,6)\) for each
endomorphism \(e\) and voicing type \(\kappa\).
Here \(\ast\) denotes all four values \(1,2,3,6\), and
\(\varnothing\) denotes none.}
\label{tab:transfer}
\end{table}

For example, the dominant--tonic and subdominant--tonic pairs
\((X_7^+,X_0^+)\) and \((X_5^+,X_0^+)\) each admit 22 of the 32 realizations,
matched by the correspondence in (a).

Thus the successor relation does not single out the dominant--tonic
pair among the Noll pairs, and the special position of this pair in
Mazzola's account has no counterpart in the model.
In particular, by Theorem~\ref{thm:transfer}(a) the relation treats a
root progression by descending fifth, such as \((X_7^+,X_0^+)\), and one
by ascending fifth, such as \((X_5^+,X_0^+)\), identically at the level
of transporter realizations, whereas grammatical accounts of functional
harmony treat the direction of root motion as significant
\citep[Chapter~7; see also Chapter~2]{TymoczkoTonality2023}.
The relation does
distinguish endomorphisms. For \(Y=X_0^+\) and every Noll pair, the
endomorphism \(e=T_4\cdot8\) is admitted in all four voicing types; for
the dominant--tonic pair it gives \(f=T_8\cdot8\), which sends G, B, and
D to E, C, and C, so the leading tone rises to the tonic. The identity
endomorphism, whose realization moves all three voices by the same
interval, is admitted for no pair of major triads.

\begin{remark}\label{rem:transfer}
The invariance is restricted to transporter realizations: in
the supplementary material, the transition from configuration
\(h_{14}\) to \(h_2\) is forbidden at displacement~1 but admitted
at displacement~5. The same-mode hypothesis in (b) is also essential:
extending the definition of \(\operatorname{Adm}\) to different modes
gives \(|\operatorname{Adm}(X_0^+,X_0^-)|=23\).
\end{remark}

\subsection{Parsimonious neo-Riemannian voice leadings}\label{subsec:parsimonious}

For \(N\in\{\mathsf P,\mathsf L,\mathsf R\}\) and \(X\in\mathcal T\), the
triads \(X\) and \(N(X)\) share two pitch classes
\citep{Cohn1997}. Let \(x\) be the pitch
class of \(X\) not in \(N(X)\), and \(y\) the pitch class of \(N(X)\) not
in \(X\). For \(\kappa\in\mathcal V\), the \emph{parsimonious
realization} of \(N\) in type \(\kappa\) is the transition
\(\xi_\kappa(X)\to\eta\) in which the voice holding \(x\) moves to \(y\)
and the other two voices stay. The target \(\eta\) lies in
\(K_{H_{\mathrm{Fux}}}\) unless its lower voice holds the fifth of
\(N(X)\). This happens in two of the four types when
\(N\in\{\mathsf L,\mathsf R\}\) and never when \(N=\mathsf P\), so 192 of
the 288 parsimonious realizations have targets in
\(K_{H_{\mathrm{Fux}}}\).

\begin{proposition}\label{prop:parsimonious}
A parsimonious realization with target in \(K_{H_{\mathrm{Fux}}}\) is
admitted if and only if \(N=\mathsf P\), \(X\) is minor, and
\(\kappa=351\). Thus 12 of the 192 are admitted. An example is
E\(\flat\), G, C \(\to\) E, G, C, listed from the lower voice up.
\end{proposition}

\begin{proof}
The two voices holding the common pitch classes form a pair \(\alpha\)
with \(\pi_\alpha(\eta)=\pi_\alpha(\xi_\kappa(X))\). If \(\alpha\)
contains the lower voice, it is active because
\(H_{\mathrm{Fux}}\subseteq K\times K\). If \(\alpha=\mathrm{MU}\), its
interval is a third or a sixth unless the common pitch classes are the
root and the fifth of \(X\), which happens only for \(N=\mathsf P\). The
interval is then 7 when the root is in the middle voice and 5 when the
fifth is, and only the second case, \(\kappa=351\), leaves
\(\mathrm{MU}\) inactive. By~\eqref{eq:stationary}, all other
cases are excluded. For \(X=X_r^+\) and \(\kappa=351\), the transition goes from
\(h_{10}\) to \(h_{15}\) with lower-voice displacement 11, which
the supplementary material forbids. For \(X=X_r^-\), it goes
from \(h_{15}\) to \(h_{10}\) with displacement 1, which is admitted.
\end{proof}

\begin{remark}\label{rem:parsimonious}
The obstruction is inherited from the two-voice model, in which a
source is never its own successor. Removing it would require a
different successor relation, which we do not pursue here.
\end{remark}

\section*{AI usage disclosure}

OpenAI ChatGPT and Codex, and Anthropic Claude, were used during the
preparation and revision of this manuscript for English-language
editing, LaTeX revision, terminology and consistency checks, and
assistance with computational verification of the successor tables.
The authors take full responsibility for the mathematical content and
have checked all AI-assisted revisions and computations.

\appendix
\section{Successor counts and a representative example}\label{app:successor-summary}

Table~\ref{tab:labels} assigns the labels \(h_1,\ldots,h_{26}\) to the
configurations of \(H_{\mathrm{Fux}}\); they are used in
Section~\ref{sec:noll} and in the supplementary material. Each label
applies to \(\begin{bmatrix}c\\b\\a\end{bmatrix}\) for every lower-voice pitch
class \(a\). The labels are indices, not pitch classes or chord names.

\begin{table}[H]
\centering
\renewcommand{\arraystretch}{1.0}
\begin{tabular}{cc@{\qquad}cc}
\toprule
Label & \(\begin{bmatrix}c\\b\end{bmatrix}\) & Label & \(\begin{bmatrix}c\\b\end{bmatrix}\) \\
\midrule
\(h_{1}\) & \(\begin{bmatrix}0\\0\end{bmatrix}\) & \(h_{14}\) & \(\begin{bmatrix}7\\4\end{bmatrix}\) \\\addlinespace[4pt]
\(h_{2}\) & \(\begin{bmatrix}3\\0\end{bmatrix}\) & \(h_{15}\) & \(\begin{bmatrix}9\\4\end{bmatrix}\) \\\addlinespace[4pt]
\(h_{3}\) & \(\begin{bmatrix}4\\0\end{bmatrix}\) & \(h_{16}\) & \(\begin{bmatrix}0\\7\end{bmatrix}\) \\\addlinespace[4pt]
\(h_{4}\) & \(\begin{bmatrix}7\\0\end{bmatrix}\) & \(h_{17}\) & \(\begin{bmatrix}3\\7\end{bmatrix}\) \\\addlinespace[4pt]
\(h_{5}\) & \(\begin{bmatrix}8\\0\end{bmatrix}\) & \(h_{18}\) & \(\begin{bmatrix}4\\7\end{bmatrix}\) \\\addlinespace[4pt]
\(h_{6}\) & \(\begin{bmatrix}9\\0\end{bmatrix}\) & \(h_{19}\) & \(\begin{bmatrix}7\\7\end{bmatrix}\) \\\addlinespace[4pt]
\(h_{7}\) & \(\begin{bmatrix}0\\3\end{bmatrix}\) & \(h_{20}\) & \(\begin{bmatrix}0\\8\end{bmatrix}\) \\\addlinespace[4pt]
\(h_{8}\) & \(\begin{bmatrix}3\\3\end{bmatrix}\) & \(h_{21}\) & \(\begin{bmatrix}3\\8\end{bmatrix}\) \\\addlinespace[4pt]
\(h_{9}\) & \(\begin{bmatrix}7\\3\end{bmatrix}\) & \(h_{22}\) & \(\begin{bmatrix}8\\8\end{bmatrix}\) \\\addlinespace[4pt]
\(h_{10}\) & \(\begin{bmatrix}8\\3\end{bmatrix}\) & \(h_{23}\) & \(\begin{bmatrix}0\\9\end{bmatrix}\) \\\addlinespace[4pt]
\(h_{11}\) & \(\begin{bmatrix}9\\3\end{bmatrix}\) & \(h_{24}\) & \(\begin{bmatrix}3\\9\end{bmatrix}\) \\\addlinespace[4pt]
\(h_{12}\) & \(\begin{bmatrix}0\\4\end{bmatrix}\) & \(h_{25}\) & \(\begin{bmatrix}4\\9\end{bmatrix}\) \\\addlinespace[4pt]
\(h_{13}\) & \(\begin{bmatrix}4\\4\end{bmatrix}\) & \(h_{26}\) & \(\begin{bmatrix}9\\9\end{bmatrix}\) \\
\bottomrule
\end{tabular}
\caption{Labels of the 26 configurations of \(H_{\mathrm{Fux}}\).}\label{tab:labels}
\end{table}
Table~\ref{tab:successor-counts} gives \(|S_{H_{\mathrm{Fux}}}(\xi)|\)
for all 26 source configurations. Each count includes the targets for
all 12 lower-voice displacements and does not depend on the pitch
class of the source lower voice. These numbers are cardinalities of
unions of maximizing target sets, not the maxima
\(M_{H_{\mathrm{Fux}}}(\xi)\).

\begin{table}[H]
\centering
\renewcommand{\arraystretch}{1.15}
\begin{tabular}{lc}
\toprule
Source configuration \(\begin{bmatrix}c\\b\end{bmatrix}\) & \(|S_{H_{\mathrm{Fux}}}(\xi)|\) \\
\midrule
\(\begin{bmatrix}8\\8\end{bmatrix}\) & 120 \\
\(\begin{bmatrix}8\\0\end{bmatrix},\ \begin{bmatrix}0\\4\end{bmatrix}\) & 126 \\
\(\begin{bmatrix}3\\3\end{bmatrix},\ \begin{bmatrix}9\\9\end{bmatrix}\) & 132 \\
\(\begin{bmatrix}7\\0\end{bmatrix},\ \begin{bmatrix}7\\3\end{bmatrix}\) & 136 \\
\(\begin{bmatrix}0\\0\end{bmatrix},\ \begin{bmatrix}4\\0\end{bmatrix},\ \begin{bmatrix}4\\4\end{bmatrix},\ \begin{bmatrix}7\\4\end{bmatrix},\ \begin{bmatrix}0\\8\end{bmatrix}\) & 138 \\
\(\begin{bmatrix}4\\9\end{bmatrix}\) & 140 \\
\(\begin{bmatrix}4\\7\end{bmatrix}\) & 146 \\
\(\begin{bmatrix}7\\7\end{bmatrix},\ \begin{bmatrix}3\\8\end{bmatrix}\) & 152 \\
\(\begin{bmatrix}3\\0\end{bmatrix},\ \begin{bmatrix}9\\0\end{bmatrix}\) & 156 \\
\(\begin{bmatrix}3\\7\end{bmatrix}\) & 168 \\
\(\begin{bmatrix}0\\3\end{bmatrix},\ \begin{bmatrix}0\\9\end{bmatrix}\) & 184 \\
\(\begin{bmatrix}0\\7\end{bmatrix}\) & 198 \\
\(\begin{bmatrix}8\\3\end{bmatrix}\) & 204 \\
\(\begin{bmatrix}9\\4\end{bmatrix}\) & 206 \\
\(\begin{bmatrix}9\\3\end{bmatrix},\ \begin{bmatrix}3\\9\end{bmatrix}\) & 256 \\
\bottomrule
\end{tabular}
\caption{Numbers of admitted successors for the harmonic mask
\(H_{\mathrm{Fux}}\).}
\label{tab:successor-counts}
\end{table}

For example, take \(\xi=\begin{bmatrix}0\\0\\0\end{bmatrix}\) and the lower-voice displacement
\(j=0\). The admitted target configurations are exactly
\[
\bigl\{\begin{bmatrix}3\\3\end{bmatrix},\begin{bmatrix}7\\3\end{bmatrix},\begin{bmatrix}3\\7\end{bmatrix},\begin{bmatrix}7\\7\end{bmatrix},\begin{bmatrix}9\\9\end{bmatrix}\bigr\}.
\]
This gives five successors whose lower-voice pitch class is 0; the
total over all displacements is 138. The complete forbidden-target
sets are listed in the supplementary material.

\end{document}


\title{Supplementary Table for:\\[0.5em]The art of counterpoint:\\
A Mazzola-type model of three-voice first-species counterpoint}
\author{
Ren Okumura\thanks{Corresponding author. Email: 25rmd09@ms.dendai.ac.jp.
This work was supported by the Research Institute for Science and
Technology of Tokyo Denki University, Japan (grant number Q26K-01).},\quad Takuro Shibayama\\
Department of Informatics, Tokyo Denki University, Tokyo, Japan
}
\date{29 September 2026}
\maketitle

\renewcommand{\thesection}{S\arabic{section}}

\section{Complete forbidden-target tables}\label{app:forbidden-targets}

For the mask \(H_{\mathrm{Fux}}\) in Section~3 of the main article, we
label the interval configurations \(\begin{bmatrix}c\\b\end{bmatrix}\) by
\(\mathcal L=\{h_1,\ldots,h_{26}\}\) as follows. Each label applies to
\(\begin{bmatrix}c\\b\\a\end{bmatrix}\) for every lower-voice pitch class \(a\). The labels are
indices, not pitch classes or chord names; they coincide with the labels
in Appendix~A of the main article.

\begin{center}
\renewcommand{\arraystretch}{1.1}
\begin{tabular}{cc@{\qquad}cc}
\toprule
Label & \(\begin{bmatrix}c\\b\end{bmatrix}\) & Label & \(\begin{bmatrix}c\\b\end{bmatrix}\) \\
\midrule
\(h_{1}\) & \(\begin{bmatrix}0\\0\end{bmatrix}\) & \(h_{14}\) & \(\begin{bmatrix}7\\4\end{bmatrix}\) \\
\(h_{2}\) & \(\begin{bmatrix}3\\0\end{bmatrix}\) & \(h_{15}\) & \(\begin{bmatrix}9\\4\end{bmatrix}\) \\
\(h_{3}\) & \(\begin{bmatrix}4\\0\end{bmatrix}\) & \(h_{16}\) & \(\begin{bmatrix}0\\7\end{bmatrix}\) \\
\(h_{4}\) & \(\begin{bmatrix}7\\0\end{bmatrix}\) & \(h_{17}\) & \(\begin{bmatrix}3\\7\end{bmatrix}\) \\
\(h_{5}\) & \(\begin{bmatrix}8\\0\end{bmatrix}\) & \(h_{18}\) & \(\begin{bmatrix}4\\7\end{bmatrix}\) \\
\(h_{6}\) & \(\begin{bmatrix}9\\0\end{bmatrix}\) & \(h_{19}\) & \(\begin{bmatrix}7\\7\end{bmatrix}\) \\
\(h_{7}\) & \(\begin{bmatrix}0\\3\end{bmatrix}\) & \(h_{20}\) & \(\begin{bmatrix}0\\8\end{bmatrix}\) \\
\(h_{8}\) & \(\begin{bmatrix}3\\3\end{bmatrix}\) & \(h_{21}\) & \(\begin{bmatrix}3\\8\end{bmatrix}\) \\
\(h_{9}\) & \(\begin{bmatrix}7\\3\end{bmatrix}\) & \(h_{22}\) & \(\begin{bmatrix}8\\8\end{bmatrix}\) \\
\(h_{10}\) & \(\begin{bmatrix}8\\3\end{bmatrix}\) & \(h_{23}\) & \(\begin{bmatrix}0\\9\end{bmatrix}\) \\
\(h_{11}\) & \(\begin{bmatrix}9\\3\end{bmatrix}\) & \(h_{24}\) & \(\begin{bmatrix}3\\9\end{bmatrix}\) \\
\(h_{12}\) & \(\begin{bmatrix}0\\4\end{bmatrix}\) & \(h_{25}\) & \(\begin{bmatrix}4\\9\end{bmatrix}\) \\
\(h_{13}\) & \(\begin{bmatrix}4\\4\end{bmatrix}\) & \(h_{26}\) & \(\begin{bmatrix}9\\9\end{bmatrix}\) \\
\bottomrule
\end{tabular}
\end{center}

In each table, \(j\) is the target lower-voice pitch class minus the
source lower-voice pitch class, modulo 12. The listed sets consist of
forbidden target labels; complements are taken in the 26-element label
set \(\mathcal L\), not in the 312-element state set
\(K_{H_{\mathrm{Fux}}}\). A label that does not appear in the forbidden
set is admitted; thus \(\varnothing\) means that all 26 target
configurations are admitted for that source and displacement. The two
displacements \(j\) and \(j+6\) in each heading have identical
forbidden sets.

Together, the tables specify the decisions for all
\(26\cdot12\cdot26=8112\) normalized (source, displacement, target)
triples. For a source with lower-voice pitch class \(a\), a target
configuration \(\begin{bmatrix}c'\\b'\end{bmatrix}\) and a displacement \(j\) determine the target
state \(\begin{bmatrix}c'\\b'\\a+j\end{bmatrix}\), with arithmetic modulo 12.

\clearpage
\noindent\begin{minipage}[t][0.48\textheight][t]{\linewidth}
\subsection{Lower-voice displacements 0 and 6}
\begin{center}
\small
\begin{adjustbox}{max width=\linewidth}
$\def\arraystretch{1.15}
\begin{array}{c|l}
\text{Source labels} & \text{Forbidden target labels}\\
\hline \hline
h_{1},h_{3},h_{13},h_{20}&\calL \setminus \left\{h_{8},h_{9},h_{17},h_{19},h_{26}\right\}
\\
\hline h_{2},h_{6}&\calL \setminus \left\{h_{9},h_{10},h_{12},h_{13},h_{14},h_{19},h_{20},h_{22}\right\}
\\
\hline h_{4}&\calL \setminus\left\{h_{7},h_{8},h_{12},h_{13},h_{17},h_{18},h_{20},h_{22}\right\}
\\
\hline h_{5},h_{12}&\calL \setminus \left\{h_{8},h_{11},h_{24},h_{26}\right\}
\\
\hline h_{7},h_{23}& \calL \setminus \left\{h_{3},h_{4},h_{5},h_{13},h_{14},h_{17},h_{18},h_{19},h_{21},h_{22}\right\}
\\
\hline h_{8},h_{26}& \calL \setminus \left\{h_{3},h_{4},h_{5},h_{13},h_{14},h_{19},h_{22}\right\}
\\
\hline h_{9}&\calL \setminus \left\{h_{1},h_{2},h_{3},h_{5},h_{6},h_{12},h_{13},h_{18},h_{20},h_{21},h_{22}\right\}
\\
\hline h_{10}&\left\{h_{5},h_{6},h_{7},h_{8},h_{9},h_{10},h_{11},h_{15},h_{22},h_{23},h_{24},h_{25},h_{26}\right\}
\\
\hline h_{11},h_{24}&\calL \setminus \left\{h_{1},h_{3},h_{4},h_{5},h_{12},h_{13},h_{14},h_{16},h_{18},h_{19},h_{20},h_{22}\right\}
\\
\hline h_{14}&\calL \setminus\left\{h_{7},h_{8},h_{16},h_{17},h_{23},h_{25},h_{26}\right\}
\\
\hline h_{15}&\calL \setminus\left\{h_{7},h_{9},h_{10},h_{16},h_{18},h_{19},h_{23},h_{25}\right\}
\\
\hline h_{16}&\calL \setminus \left\{h_{2},h_{4},h_{6},h_{8},h_{9},h_{11},h_{14},h_{15},h_{21},h_{24},h_{26}\right\}
\\
\hline h_{17}&\calL \setminus\left\{h_{1},h_{3},h_{4},h_{9},h_{13},h_{14},h_{20},h_{22},h_{23},h_{25}\right\}
\\
\hline h_{18}&\calL \setminus \left\{h_{2},h_{4},h_{6},h_{8},h_{9},h_{15},h_{21},h_{26}\right\}
\\
\hline h_{19}&\left\{h_{4},h_{6},h_{7},h_{9},h_{10},h_{13},h_{14},h_{15},h_{16},h_{17},h_{18},h_{19}\right\}
\\
\hline h_{21}&\calL \setminus \left\{h_{1},h_{3},h_{5},h_{7},h_{9},h_{12},h_{13},h_{14},h_{18},h_{19}\right\}
\\
\hline h_{22}&\calL \setminus \left\{h_{7},h_{8},h_{9},h_{17},h_{18},h_{19},h_{23},h_{25}\right\}
\\
\hline h_{25}&\calL \setminus\left\{h_{2},h_{6},h_{14},h_{17},h_{19}\right\}
\end{array}
$
\end{adjustbox}
\end{center}
\end{minipage}
\par\vfill
\noindent\begin{minipage}[t][0.48\textheight][t]{\linewidth}
\subsection{Lower-voice displacements 1 and 7}
\begin{center}
\small
\begin{adjustbox}{max width=\linewidth}
$\def\arraystretch{1.15}
\begin{array}{c|l}
\text{Source labels} & \text{Forbidden target labels}\\
\hline \hline 
h_{1},h_{3},h_{13},h_{20}&\left\{h_{8},h_{9},h_{10},h_{15},h_{16},h_{17},h_{19},h_{26}\right\}
\\
\hline h_{2},h_{6}&\left\{h_{6},h_{7},h_{10},h_{11},h_{15},h_{16},h_{23},h_{24},h_{25},h_{26}\right\}
\\
\hline h_{4}&\left\{h_{4},h_{9},h_{10},h_{11},h_{14},h_{15},h_{16},h_{19},h_{21},h_{23},h_{24},h_{25},h_{26}\right\}
\\
\hline h_{5},h_{12}&\left\{h_{5},h_{6},h_{7},h_{10},h_{12},h_{15},h_{16},h_{17},h_{18}\right\}
\\
\hline h_{7},h_{23}&\left\{h_{6},h_{10},h_{11},h_{15},h_{16},h_{18},h_{24},h_{26}\right\}
\\
\hline h_{8},h_{26}&\left\{h_{5},h_{6},h_{7},h_{10},h_{11},h_{15},h_{16},h_{18},h_{22},h_{23},h_{24},h_{25},h_{26}\right\}
\\
\hline h_{9}&\calL \setminus \left\{h_{2},h_{6},h_{7},h_{8},h_{17},h_{18},h_{21},h_{23},h_{25},h_{26}\right\}
\\
\hline h_{10}&\left\{h_{5},h_{6},h_{10},h_{11},h_{15},h_{22},h_{26}\right\}
\\
\hline h_{11},h_{15},h_{24}&\varnothing
\\
\hline h_{14}&\left\{h_{2},h_{4},h_{6},h_{7},h_{9},h_{11},h_{14},h_{15},h_{16},h_{19},h_{23},h_{24}\right\}
\\
\hline h_{16}&\left\{h_{16},h_{17},h_{18},h_{19}\right\}
\\
\hline h_{17}&\left\{h_{5},h_{6},h_{7},h_{10},h_{12},h_{15},h_{16},h_{17},h_{18},h_{19}\right\}
\\
\hline h_{18}&\left\{h_{2},h_{6},h_{7},h_{11},h_{15},h_{16},h_{17},h_{18},h_{19},h_{23},h_{24}\right\}
\\
\hline h_{19}&\calL \setminus\left\{h_{2},h_{3},h_{5},h_{12},h_{13},h_{20},h_{21},h_{22},h_{23},h_{25}\right\}
\\
\hline h_{21}&\left\{h_{4},h_{10},h_{11},h_{15},h_{16},h_{20},h_{21},h_{22},h_{23},h_{24},h_{25},h_{26}\right\}
\\
\hline h_{22}&\calL \setminus\left\{h_{1},h_{2},h_{3},h_{4},h_{7},h_{12},h_{13},h_{14},h_{20},h_{21},h_{23},h_{24}\right\}
\\
\hline h_{25}&\left\{h_{4},h_{10},h_{11},h_{15},h_{16},h_{21},h_{24},h_{25}\right\}
\end{array}
$
\end{adjustbox}
\end{center}
\end{minipage}
\clearpage
\noindent\begin{minipage}[t][0.48\textheight][t]{\linewidth}
\subsection{Lower-voice displacements 2 and 8}
\begin{center}
\small
\begin{adjustbox}{max width=\linewidth}
$\def\arraystretch{1.15}
\begin{array}{c|l}
\text{Source labels} & \text{Forbidden target labels}\\
\hline \hline 
h_{1},h_{3},h_{13},h_{20}&\calL \setminus \left\{h_{8},h_{9},h_{17},h_{19},h_{26}\right\}
\\
\hline h_{2},h_{6}&\left\{h_{4},h_{9},h_{10},h_{11},h_{14},h_{18},h_{19},h_{23},h_{24},h_{25},h_{26}\right\}
\\
\hline h_{4}&\left\{h_{4},h_{9},h_{10},h_{11},h_{14},h_{15},h_{16},h_{19},h_{21},h_{23},h_{24},h_{25},h_{26}\right\}
\\
\hline h_{5},h_{12}&\calL \setminus \left\{h_{8},h_{11},h_{24},h_{26}\right\}
\\
\hline h_{7},h_{23}&\left\{h_{6},h_{10},h_{11},h_{14},h_{15},h_{18},h_{24},h_{26}\right\}
\\
\hline h_{8},h_{26}&\left\{h_{3},h_{4},h_{6},h_{7},h_{9},h_{10},h_{13},h_{14},h_{15},h_{16},h_{17},h_{18},h_{19}\right\}
\\
\hline h_{9}&\left\{h_{4},h_{8},h_{9},h_{10},h_{14},h_{15},h_{16},h_{17},h_{19},h_{26}\right\}
\\
\hline h_{10}&\left\{h_{5},h_{6},h_{10},h_{11},h_{15},h_{22},h_{26}\right\}
\\
\hline h_{11},h_{24}&\varnothing
\\
\hline h_{14}&\calL \setminus \left\{h_{7},h_{8},h_{17},h_{18},h_{23},h_{25},h_{26}\right\}
\\
\hline h_{15}&\left\{h_{1},h_{2},h_{3},h_{4},h_{5},h_{6},h_{12},h_{13},h_{14},h_{15},h_{20},h_{21},h_{22}\right\}
\\
\hline h_{16}&\calL \setminus\left\{h_{2},h_{4},h_{6},h_{8},h_{9},h_{11},h_{14},h_{15},h_{21},h_{24},h_{26}\right\}
\\
\hline h_{17}&\left\{h_{5},h_{6},h_{7},h_{10},h_{12},h_{15},h_{16},h_{17},h_{18},h_{19}\right\}
\\
\hline h_{18}&\calL \setminus \left\{h_{2},h_{4},h_{6},h_{8},h_{9},h_{14},h_{21},h_{26}\right\}
\\
\hline h_{19}&\left\{h_{4},h_{6},h_{7},h_{9},h_{10},h_{14},h_{15},h_{16},h_{17},h_{18},h_{19},h_{22}\right\}
\\
\hline h_{21}&\left\{h_{4},h_{10},h_{11},h_{15},h_{16},h_{20},h_{21},h_{22},h_{23},h_{24},h_{25},h_{26}\right\}
\\
\hline h_{22}&\calL \setminus\left\{h_{7},h_{8},h_{9},h_{17},h_{18},h_{19},h_{23},h_{25}\right\}
\\
\hline h_{25}&\calL \setminus\left\{h_{2},h_{6},h_{8},h_{9},h_{14},h_{17},h_{19},h_{26}\right\}
\end{array}
$
\end{adjustbox}
\end{center}
\end{minipage}
\par\vfill
\noindent\begin{minipage}[t][0.48\textheight][t]{\linewidth}
\subsection{Lower-voice displacements 3 and 9}
\begin{center}
\small
\begin{adjustbox}{max width=\linewidth}
$\def\arraystretch{1.15}
\begin{array}{c|l}
\text{Source labels} & \text{Forbidden target labels}\\
\hline \hline 
h_{1},h_{3},h_{13},h_{20}
&
\left\{h_{8},h_{9},h_{10},h_{15},h_{16},h_{17},h_{19},h_{26}\right\}
\\
\hline h_{2},h_{6}
&
\calL \setminus\left\{h_{9},h_{10},h_{12},h_{13},h_{14},h_{19},h_{20},h_{22}\right\}
\\
\hline h_{4}
&
\calL \setminus \left\{h_{7},h_{8},h_{12},h_{13},h_{17},h_{18},h_{20},h_{22}\right\}
\\
\hline h_{5},h_{12}
&
\left\{h_{5},h_{6},h_{7},h_{10},h_{12},h_{15},h_{16},h_{17},h_{18}\right\}
\\
\hline h_{7},h_{23}
&
\calL \setminus \left\{h_{3},h_{4},h_{5},h_{13},h_{14},h_{17},h_{18},h_{19},h_{21},h_{22}\right\}
\\
\hline h_{8},h_{26}
&
\calL \setminus \left\{h_{3},h_{4},h_{5},h_{13},h_{14},h_{19},h_{22}\right\}
\\
\hline h_{9}
&
\calL \setminus \left\{h_{2},h_{6},h_{17},h_{18},h_{21}\right\}
\\
\hline h_{10}
&
\left\{h_{5},h_{6},h_{7},h_{8},h_{9},h_{10},h_{11},h_{15},h_{22},h_{23},h_{24},h_{25},h_{26}\right\}
\\
\hline h_{11},h_{24}
&
\calL \setminus \left\{h_{1},h_{3},h_{4},h_{5},h_{12},h_{13},h_{14},h_{16},h_{18},h_{19},h_{20},h_{22}\right\}
\\
\hline h_{14}
&
\left\{h_{4},h_{9},h_{10},h_{11},h_{14},h_{18},h_{19},h_{24}\right\}
\\
\hline h_{15}
&
\left\{h_{2},h_{6},h_{8},h_{11},h_{15},h_{17},h_{21},h_{24},h_{26}\right\}
\\
\hline h_{16}
&
\left\{h_{16},h_{17},h_{18},h_{19}\right\}
\\
\hline h_{17}
&
\calL \setminus \left\{h_{1},h_{3},h_{4},h_{9},h_{13},h_{14},h_{20},h_{22},h_{23},h_{25}\right\}
\\
\hline h_{18}
&
\left\{h_{10},h_{11},h_{14},h_{16},h_{17},h_{18},h_{19},h_{24}\right\}
\\
\hline h_{19}
&
\left\{h_{4},h_{6},h_{7},h_{9},h_{10},h_{13},h_{14},h_{15},h_{16},h_{17},h_{18},h_{19}\right\}
\\
\hline h_{21}
&
\calL \setminus \left\{h_{1},h_{3},h_{5},h_{7},h_{9},h_{12},h_{13},h_{14},h_{18},h_{19}\right\}
\\
\hline h_{22}
&
\calL \setminus \left\{h_{1},h_{2},h_{3},h_{4},h_{7},h_{12},h_{13},h_{14},h_{20},h_{21},h_{23},h_{24}\right\}
\\
\hline h_{25}
&
\left\{h_{4},h_{7},h_{8},h_{9},h_{10},h_{11},h_{15},h_{16},h_{21},h_{23},h_{24},h_{25},h_{26}\right\}
\end{array}
$
\end{adjustbox}
\end{center}
\end{minipage}
\clearpage
\noindent\begin{minipage}[t][0.48\textheight][t]{\linewidth}
\subsection{Lower-voice displacements 4 and 10}
\begin{center}
\small
\begin{adjustbox}{max width=\linewidth}
$\def\arraystretch{1.15}
\begin{array}{c|l}
\text{Source labels} & \text{Forbidden target labels}\\
\hline \hline 
h_{1},h_{3},h_{13},h_{20}
&
\calL \setminus \left\{h_{8},h_{9},h_{17},h_{19},h_{26}\right\}
\\
\hline h_{2},h_{6}
&
\left\{h_{6},h_{7},h_{10},h_{11},h_{15},h_{16},h_{23},h_{24},h_{25},h_{26}\right\}
\\
\hline h_{4}
&
\left\{h_{4},h_{9},h_{10},h_{11},h_{14},h_{15},h_{16},h_{19},h_{21},h_{23},h_{24},h_{25},h_{26}\right\}
\\
\hline h_{5},h_{12}
&
\calL \setminus \left\{h_{8},h_{11},h_{24},h_{26}\right\}
\\
\hline h_{7},h_{23}
&
\left\{h_{6},h_{10},h_{11},h_{15},h_{16},h_{18},h_{24},h_{26}\right\}
\\
\hline h_{8},h_{26}
&
\left\{h_{5},h_{6},h_{7},h_{10},h_{11},h_{15},h_{16},h_{18},h_{22},h_{23},h_{24},h_{25},h_{26}\right\}
\\
\hline h_{9}
&
\left\{h_{4},h_{8},h_{9},h_{10},h_{14},h_{15},h_{16},h_{17},h_{19},h_{26}\right\}
\\
\hline h_{10}
&
\left\{h_{5},h_{6},h_{10},h_{11},h_{15},h_{22},h_{26}\right\}
\\
\hline h_{11},h_{24}
&
\varnothing
\\
\hline h_{14}
&
\calL \setminus \left\{h_{8},h_{10},h_{17},h_{18},h_{25},h_{26}\right\}
\\
\hline h_{15}
&
\left\{h_{1},h_{2},h_{3},h_{4},h_{5},h_{6},h_{12},h_{13},h_{14},h_{15},h_{20},h_{21},h_{22}\right\}
\\
\hline h_{16}
&
\calL \setminus \left\{h_{2},h_{4},h_{6},h_{8},h_{9},h_{11},h_{14},h_{15},h_{21},h_{24},h_{26}\right\}
\\
\hline h_{17}
&
\left\{h_{5},h_{6},h_{7},h_{10},h_{12},h_{15},h_{16},h_{17},h_{18},h_{19}\right\}
\\
\hline h_{18}
&
\calL \setminus \left\{h_{4},h_{8},h_{9},h_{14},h_{21},h_{26}\right\}
\\
\hline h_{19}
&
\calL \setminus \left\{h_{2},h_{3},h_{5},h_{12},h_{13},h_{20},h_{21},h_{22},h_{23},h_{25}\right\}
\\
\hline h_{21}
&
\left\{h_{4},h_{10},h_{11},h_{15},h_{16},h_{20},h_{21},h_{22},h_{23},h_{24},h_{25},h_{26}\right\}
\\
\hline h_{22}
&
\calL \setminus \left\{h_{7},h_{8},h_{9},h_{17},h_{18},h_{19},h_{23},h_{25}\right\}
\\
\hline h_{25}
&
\calL \setminus \left\{h_{2},h_{6},h_{8},h_{9},h_{14},h_{17},h_{19},h_{26}\right\}
\end{array}
$
\end{adjustbox}
\end{center}
\end{minipage}
\par\vfill
\noindent\begin{minipage}[t][0.48\textheight][t]{\linewidth}
\subsection{Lower-voice displacements 5 and 11}
\begin{center}
\small
\begin{adjustbox}{max width=\linewidth}
$\def\arraystretch{1.15}
\begin{array}{c|l}
\text{Source labels} & \text{Forbidden target labels}\\
\hline \hline 
h_{1},h_{3},h_{13},h_{20}
&
\left\{h_{8},h_{9},h_{10},h_{15},h_{16},h_{17},h_{19},h_{26}\right\}
\\
\hline h_{2},h_{6}
&
\left\{h_{4},h_{9},h_{10},h_{11},h_{14},h_{18},h_{19},h_{23},h_{24},h_{25},h_{26}\right\}
\\
\hline h_{4}
&
\left\{h_{4},h_{9},h_{10},h_{11},h_{14},h_{15},h_{16},h_{19},h_{21},h_{23},h_{24},h_{25},h_{26}\right\}
\\
\hline h_{5},h_{12}
&
\left\{h_{5},h_{6},h_{7},h_{10},h_{12},h_{15},h_{16},h_{17},h_{18}\right\}
\\
\hline h_{7},h_{23}
&
\left\{h_{6},h_{10},h_{11},h_{14},h_{15},h_{18},h_{24},h_{26}\right\}
\\
\hline h_{8},h_{26}
&
\left\{h_{3},h_{4},h_{6},h_{7},h_{9},h_{10},h_{13},h_{14},h_{15},h_{16},h_{17},h_{18},h_{19}\right\}
\\
\hline h_{9}
&
\calL \setminus \left\{h_{2},h_{6},h_{7},h_{8},h_{17},h_{18},h_{21},h_{23},h_{25},h_{26}\right\}
\\
\hline h_{10}
&
\left\{h_{5},h_{6},h_{10},h_{11},h_{15},h_{22},h_{26}\right\}
\\
\hline h_{11},h_{15},h_{24}
&
\varnothing
\\
\hline h_{14}
&
\left\{h_{4},h_{9},h_{10},h_{11},h_{14},h_{15},h_{16},h_{19},h_{24}\right\}
\\
\hline h_{16}
&
\left\{h_{16},h_{17},h_{18},h_{19}\right\}
\\
\hline h_{17}
&
\left\{h_{5},h_{6},h_{7},h_{10},h_{12},h_{15},h_{16},h_{17},h_{18},h_{19}\right\}
\\
\hline h_{18}
&
\left\{h_{10},h_{11},h_{15},h_{16},h_{17},h_{18},h_{19},h_{24}\right\}
\\
\hline h_{19}
&
\left\{h_{4},h_{6},h_{7},h_{9},h_{10},h_{14},h_{15},h_{16},h_{17},h_{18},h_{19},h_{22}\right\}
\\
\hline h_{21}
&
\left\{h_{4},h_{10},h_{11},h_{15},h_{16},h_{20},h_{21},h_{22},h_{23},h_{24},h_{25},h_{26}\right\}
\\
\hline h_{22}
&
\calL \setminus \left\{h_{1},h_{2},h_{3},h_{4},h_{7},h_{12},h_{13},h_{14},h_{20},h_{21},h_{23},h_{24}\right\}
\\
\hline h_{25}
&
\left\{h_{4},h_{10},h_{11},h_{15},h_{16},h_{21},h_{24},h_{25}\right\}
\end{array}
$
\end{adjustbox}
\end{center}
\end{minipage}